\documentclass[18pt,oneside,reqno]{amsart}
\usepackage{amssymb}
\usepackage{txfonts}
\usepackage{bbm}
\usepackage{cases}
\usepackage{amsmath}
\usepackage{extarrows}
\usepackage{stfloats}
\usepackage{graphicx}
\usepackage{mathrsfs}
\usepackage[all]{xy}
\usepackage{stmaryrd}
\usepackage{amsfonts}
\usepackage{diagbox}
\usepackage{xcolor}
\usepackage{enumerate,amsmath,amssymb,amsthm}
\usepackage{tikz}

\usepackage{pgflibrarysnakes}
\usetikzlibrary{shapes,arrows}
\usepackage{caption,subcaption}

\makeatletter
\@addtoreset{figure}{section}
\makeatother

\oddsidemargin=0mm \topmargin=-12mm \numberwithin{equation}{section}

\numberwithin{equation}{section}
\newtheorem{thm}{Theorem}[section]
\newtheorem{proposition}{Proposition}[section]
\newtheorem{corollary}{Corollary}[section]
\newtheorem{lemma}{Lemma}[section]
\newtheorem{definition}{Definition}[section]
\newtheorem{hypothesis}{Hypothesis}[section]
\newtheorem{Rem}{Remark}[section]
\newtheorem{exmp}{Example}[section]

\def\v{{\mathbf{v}}}

\def\eps{\varepsilon}

\def\ti{\tilde}

\def\<{{\langle}}
\def\>{{\rangle}}
\def\({{\Big(}}
\def\){{\Big)}}
\def\]{{\Big]}}
\def\[{{\Big[}}

\def\bx{{\mathbf{x}}}

\def\dif{{\mathord{{\rm d}}}}

\def\={&\!\!=\!\!&}
\def\bt{\begin{theorem}}
\def\et{\end{theorem}}
\def\bl{\begin{lemma}}
\def\el{\end{lemma}}
\def\br{\begin{remark}}
\def\er{\end{remark}}

\def\bd{\begin{definition}}
\def\ed{\end{definition}}
\def\bp{\begin{proposition}}
\def\ep{\end{proposition}}
\def\bc{\begin{corollary}}
\def\ec{\end{corollary}}
\def\bx{\begin{Examples}}
\def\ex{\end{Examples}}
\def\cA{{\mathcal A}}
\def\cB{{\mathcal B}}
\def\cC{{\mathcal C}}
\def\cD{{\mathcal D}}
\def\cE{{\mathcal E}}
\def\cF{{\mathcal F}}

\def\cJ{{\mathcal J}}
\def\cK{{\mathcal K}}
\def\cL{{\mathcal L}}
\def\cM{{\mathcal M}}
\def\cN{{\mathcal N}}
\def\cO{{\mathcal O}}

\def\cS{{\mathcal S}}

\def\cX{{\mathcal X}}
\def\cY{{\mathcal Y}}
\def\cZ{{\mathcal Z}}

\def\mE{{\mathbb E}}

\def\mN{{\mathbb N}}

\def\mP{{\mathbb P}}

\def\mR{{\mathbb R}}

\def\mT{{\mathbb T}}

\def\mZ{{\mathbb Z}}

\def\geq{\geqslant}
\def\leq{\leqslant}

\numberwithin{equation}{section}

\title{\bf{Exponential mixing for the fractional Magneto-Hydrodynamic equations with degenerate stochastic forcing$^*$}}

{\author{   Xuhui Peng$^{\flat,\spadesuit}$  \\
{\em\tiny  $^\flat$MOE-LCSM, School of Mathematics and Statistics, Hunan Normal University,
 }\\
 {\em\tiny  Changsha, {\rm 410081}, P. R. China}\\
 {\em\tiny  $^\spadesuit$Key Laboratory of Applied Statistics and Data Science, Hunan Normal University
 }\\
  {\em\tiny College of Hunan Province,  Changsha, {\rm 410081}, P. R. China}\\
 {  Jianhua Huang$^\dag$ ,\quad Yan Zheng }\\
{\em\tiny   College of Liberal Arts and Sciences, National  University of Defense Technology,
 }\\
{\em\tiny  Changsha, {\rm 410073}, P.R.China.}\\
}}

\date{}
\begin{document}
\footnote{$^*$
The first author was supported by Hunan Provincial Natural Science Foundation of China
(No.2019JJ50377), NSFC (No.11871476) and the Construct Program of the Key Discipline in Hunan Province.
The last two were  supported  by the
 NSF of China(No.11771449).}
\footnote{$^\dag$ Corresponding author: jhhuang32@nudt.edu.cn.}

\maketitle
\begin{abstract}
We establish the existence, uniqueness and exponential attraction properties of an invariant measure for the MHD equations with degenerate stochastic forcing acting only in the  magnetic equation. The central challenge is to establish time asymptotic smoothing properties of the associated Markovian semigroup corresponding to this system. Towards this aim we take full advantage of the characteristics of the  advective structure to discover a novel H$\mathrm{\ddot{o}}$rmander-type condition which only allows for  several noises in   the  magnetic direction.

\vspace{1mm}
\end{abstract}

\section{Introduction}
\label{}

The dynamics of the velocity and the magnetic  field in electrically conducting fluids and basic physics conservation laws can be described by the Magneto-Hydrodynamic(MHD) equations (c.f. \cite{CA,CO}). The existence, uniqueness, regularity and stability of the MHD equations have been extensive studied in many papers, see  \cite{CW,DL,LZ,ST}.

Recently there has been mounting interest in the generalized  fractional  MHD equations,
\begin{eqnarray}\label{p-39}
\left\{
  \begin{split}
     & \partial_t  u+\big[ u\cdot \nabla u+\mu (-\Delta )^{\alpha} u \big]=\big[-\nabla p+b\cdot \nabla b\big],
     \\      & \partial_t b+\big[ u\cdot \nabla b+\nu (-\Delta )^{\beta} b \big]=b\cdot  \nabla u,
     \\ & u(0,x)=u_0, b(0,x)=b_0,
  \end{split}
  \right.
\end{eqnarray}
where $u(t,x), b(t,x)\in \mR^2. $
Wu \cite{WJ} proved that equations  (\ref{p-39}) have  a unique weak solution
when $\mu>0,\nu>0, \alpha\geq \frac{1}{2}+\frac{d}{4},  \beta\geq \frac{1}{2}+\frac{d}{4}, (u_0,b_0)\in L^2$  and equations  (\ref{p-39}) have  a classical   solution
when $\mu>0,\nu>0, \alpha\geq \frac{1}{2}+\frac{d}{4},  \beta\geq \frac{1}{2}+\frac{d}{4},$ and $u_0,b_0$ are sufficiently smooth.
For the 2D incompressible  MHD equations with horizontal dissipation and horizontal magnetic diffusion,
Cao, Dipendra and Wu \cite{CDW} proved that equations possesses a global regular solution if $\eps,\delta>0$ and  $u_0,b_0$ are sufficiently smooth.

Meanwhile, for the MHD equations driven by non-degenerate stochastic forcing terms both in the velocity and in the magnetic field, the existence and uniqueness of invariant measure was obtained via coupling method in \cite{BDP}. Huang and Shen \cite{HS}
  proved the  well-posedness and the existence of a random attractor for the stochastic 2D incompressible fractional MHD equations driven by Gaussian multiplicative noise.
  For the stochastic fractional MHD equations with degenerate multiplicative noise on the Torus $\mT^2,$  Shen, Huang and Zeng \cite{SHZ} proved the existence and uniqueness of the invariant measure for the associated transition semigroup. The noise in \cite{SHZ} is degenerate in the sense that it drives the system only in the first finite Fourier modes.

In this paper, we consider the following MHD  equations driven by degenerate additive noise   on two-dimensional torus $\mT^2=[-\pi,\pi]^2,$
  \begin{eqnarray}\label{p-1}
  \left\{
  \begin{split}
     & \dif u+\big[ u\cdot \nabla u+(-\Delta )^{\alpha} u \big]\dif t=\big[-\nabla p+b\cdot \nabla b\big]\dif t,
     \\      & \dif b+\big[ u\cdot \nabla b+(-\Delta )^{\beta} b \big]\dif t=b\cdot  \nabla u\dif t
  +\sum_{k=(k_1,k_2)\in Z_0}\left(\frac{k_2}{|k|},-\frac{k_1}{|k|}\right)^T\alpha_k^0\cos(k\cdot x) \dif W^{k,0}
  \\ & \quad \quad \quad \quad \quad \quad  \quad \quad \quad \quad \quad   +
     \sum_{k=(k_1,k_2)\in Z_0}\left(-\frac{k_2}{|k|},\frac{k_1}{|k|}\right)^T\alpha_k^1\sin(k\cdot x) \dif W^{k,1},
      \\
      &\nabla \cdot u=\nabla \cdot b=0,
      \\ &u(0,x)=u_0(x), ~~b(0,x)=b_0(x)
      \end{split}
      \right.
  \end{eqnarray}
  with periodic boundary value conditions
  \begin{eqnarray*}
    u_i(x,t)=u_i(x+2\pi j,t), ~~ b_i(x,t)=b_i(x+2\pi j,t),~~i=1,2,
  \end{eqnarray*}
  where $\alpha>1,\beta>1,t\geq 0, j\in \mZ,$ $u=(u_1,u_2)$ and $b=(b_1,b_2)$ denote the velocity field and magnetic field respectively, $p$ is a scalar pressure, $Z_0$ is a subset of  $\mZ^2 \backslash \{0,0\}$, $(W^{k,m})_{k\in Z_0,m\in \{0,1\}}$ is a $2|Z_0|$-dimensional Brownian motion defined relative to a filtered probability space $(\Omega,\cF,\{\cF_t\}_{t\geq 0},\mP)$, $\{\alpha_k^m\}_{k\in Z_0,m\in \{0,1\}}$ are non-zero constants. Throughout this paper, we assume that $d:=2| Z_0|<\infty.$

For $n\geq 0,$  define
\begin{eqnarray}\label{p-22}
\begin{split}
  \cZ_0:&=\{k  ~\big|~  k\in  Z_0 \text{ or } -k\in  Z_0\},
 \\ \cZ_n:&=\{ k+ \ell  ~\big|~ k\in \cZ_{n-1}, \ell \in \cZ_0, \langle k,\ell^{\bot}\rangle\neq0, |k|\neq |\ell| \},
 \end{split}
\end{eqnarray}
where  $\ell^{\bot}=(-\ell_2,\ell_1)$ for any $\ell=(\ell_1,\ell_2)$ and  $\langle \cdot,\cdot \rangle$ denotes the inner product on $\mR^2.$

\begin{hypothesis}\label{p-21}
 \begin{eqnarray*}
   \cup_{k=0}^{\infty}\cZ_{2k}=\mZ^2 \backslash \{0,0\},\quad\quad
   \cup_{k=0}^{\infty}\cZ_{2k+1}=\mZ^2 \backslash \{0,0\}.
 \end{eqnarray*}
\end{hypothesis}
Hypothesis \ref{p-21} may not seem intuitive, however it falls into the most interesting case of degenerate noise\textemdash"hypoellipticity" setting, and includes many interesting examples, see Remark \ref{rem hypo} and Example \ref{exmp hypo}. Now we state our main result as below.
\begin{thm}\label{theo 1}
Assume that Hypothesis \ref{p-21} holds, then the associated Markov semigroup corresponding to \ref{p-1} possesses a unique, exponentially mixing invariant measure.  Furthermore, a law of large numbers together with a central limit theorem is established under the current circumstances.
\end{thm}
We remark that Theorem \ref{theo 1} is a simplified version of Theorem \ref{p-27} and refer readers to Section 2.2 for more details.

Nowadays ergodicity research on infinite-dimensional systems driven by degenerate stochastic forcing has attracted considerable attention (\cite{ADX,weinan,FGRT,martin,Hairer02,Hairer,K-S,KS,KS2002,Matt2002,RZ}), not only because this poses many interesting mathematical challenges, but also provides rigorous justification  for the explicit or implicit statistical measurement assumptions invoked in a physical environment. It is exciting that recently there have been remarkable breakthroughs (c.f. \cite{FGRT,martin,Hairer}), initiating the development a theory of "hypoellipticity" for degenerated forced infinite-dimensional stochastic systems. However, the whole theory is far from mature and remains in an involved formation. The reason for this is unlike in the case of finite-dimensional systems, the invertibility of Malliavin matrix is hard to prove, not to mention characterise its range. Experts have thus devised a tactful strategy to take full advantage of the structure of turbulent systems. Roughly speaking, infinite-dimensional as these systems are, their unstable directions are confined to be finitely many, and it is reasonable that one just focus on proving the Malliavin matrix to possess small eigenvalues on some spanning cones.

The technical difficulties of this method lie in how to generate successively larger finite dimensional spaces through the interaction between the nonlinear and stochastic terms and how to exert delicate spectral analysis on these spaces. To be more specific, one digging into the technical details will find that the proof virtually relies heavily on the results of progressive computation of Lie brackets using constant vector fields and nonlinear terms, by virtue of which the whole involved arguing process will be decomposed in an inductive manner and most importantly an appropriate H$\mathrm{\ddot{o}}$rmander condition will thus be determined. It is worth emphasizing that finding out such an ideal collection of Lie brackets to accomplish the task is case-by-case, there is no general recipe for all. For instance, Navier-Stokes equations and Boussinesq equations are treated quite differently and therefore lead to different H$\mathrm{\ddot{o}}$rmander condition (c.f. \cite{FGRT,martin,Hairer}).

The main contribution of the manuscript is, we successfully devised a special pattern of Lie bracket computations suitable for the fractional MHD equation, and thus propose a novel H$\mathrm{\ddot{o}}$rmander condition. Apart from  \cite{FGRT,martin}, the considered fractional MHD equations are of original formation instead of vorticity formation. Furthermore,  Due  to the  special form of   stochastic fractional   Magneto-Hydrodynamic equations  (\ref{p-1}),  we exert a series of Lie bracket computation strategically to exploit the distinctive  structure of nonlinear advective terms. Roughly speaking, we activate the noise term within the magnetic equation to spread to the velocity equation through advection, then perturb it again with stochastic forcing to generate new directions in the $b$ component of the phase space. On the flip side, new $\v$ directions can be generated similarly except being stochastically driven once. This procedure can be repeated iteratively so as to span the whole phase space as long as Hypothesis \ref{p-21} is satisfied (c.f. Section 4 for more details). Attentive readers may find that the whole deductive process and derived H$\mathrm{\ddot{o}}$rmander condition distinguish from that within \cite{FGRT,martin,Hairer}.

We would also like to add that the degenerate noise in \cite{SH2017,SHZ} is in a fairly simple manner and belongs to the so-called "essentially elliptic" setting. More precisely, although driven modes are assumed to be finite, they are forced to be one by one and required to be sufficiently many,  while in this paper we adopt a hypoellipticity setting, which allows for a limit number of directions to be driven on and off. We will further exemplify this essential difference with Example \ref{exmp hypo}, which also exhibits a distinct picture compared with \cite{FGRT,martin,Hairer}. All in all, our analysis gets the utmost out of existing techniques in the recent works but yields something peculiar and we believe it will enrich ergodic research upon systems of SPDEs.



This article is organized as follows: In Section 2 we introduce general definitions and formulate our main result (Theorem \ref{p-27}). Section 3 is devoted to some moment estimates which will be used frequently. In Section 4 we illustrate progressive computations of Lie brackets in detail. Then in Section 5 we focus on proving the  spectral properties of Malliavin matrix (Theorem \ref{p-9}) and give a
gradient estimate of  the Markov semigroup (Proposition \ref{p-28}). Finally, we provide a proof of Theorem  \ref{p-27} in Section 6.

\section{ Preliminaries}

\subsection{Mathematical setting}
In this section we introduce a functional setting for the equations \eqref{p-1}. Then we describe specifically the stochastic forcing, and thus formulate \eqref{p-1} as an abstract stochastic evolution equation. Finally, we introduce some basic elements of the Malliavin calculus centering on the Malliavin matrix.

The higher order Sobolev spaces are denoted by
\begin{eqnarray*}
  H_1^s:=\left\{u=(u_1,u_2)\in (W^{s,2}(\mT^2))^2: \nabla \cdot u=0, \int_{\mT^2}u_1(x)\dif x=\int_{\mT^2}u_2(x)\dif x=0\right\} \text{ for any }s\geq 0,
\end{eqnarray*}
where $W^{s,2}(\mT^2)$ is classical Sobolev-Slobodeckii space, and $H^{s}_1$ is equipped with the norm
\begin{eqnarray*}
  \|u\|_{H_1^s}^2:=\|u_1\|^2_{W^{s,2}}+\|u_2\|^2_{W^{s,2}},
\end{eqnarray*}
here $u=(u_1,u_2)$. Let $H_2^s=H_1^s,H^s=H_1^s\times H_2^s$ and $H_2^s$ is equipped with the same norm with $H_1^s$.
For any $U=(u,b)\in H^s,$ the norm of $U$  on the space $H^s$ is given by
\begin{eqnarray*}
  \|U\|_{H^s}^2=\|u\|_{H_1^s}^2+\|b\|_{H_2^s}^2.
\end{eqnarray*}
We also denote $H^{-s}:=(H^s)^*$ the dual space to $H^s$.
Specially,
\begin{eqnarray*}
  H_1:=H_1^0=\bigg\{u=(u_1,u_2)\in (L^2(\mT^2))^2~\big|\nabla \cdot  u=0, \int_{\mT^2}u_1(x)\dif x=\int_{\mT^2}u_2(x)\dif x=0\bigg\}.
\end{eqnarray*}
The norm on the space $H_1$ is given by
\begin{eqnarray*}
\|u\|^2=\|(u_1,u_2)\|^2:=\|u_1\|_{L^2}^2+\|u_2\|^2_{L^2}.
\end{eqnarray*}
Likewise, let $H:=H^0$. By a slight abuse of notation, $\langle \cdot , \cdot  \rangle$ may denote the inner product on  Hilbert space  $H$ or $H_1$. Let $\Pi$ be the projection operator from $(L^2(\mT^2))^2$ to the space $H_1$.

For any  $u\in H_1^{\alpha}$, let  $\Lambda^{\alpha} u=(-\Delta)^{\alpha/2}u.$
For any  $b\in H_2^{\beta}$, let  $\Lambda^{\beta} b=(-\Delta)^{\beta/2}b.$

For any  $m,n\in \mR,$ we denote by
\begin{eqnarray*}
H^{m; n}=\big\{w=(u,b)~\big|~u\in H_1^m,b\in H_2^n\big\},
\end{eqnarray*}
endowed with the norm
$
\|w\|_{H^{m;n}}^2=\|u\|_{H_1^m}^2+\|b\|_{H_2^n}^2.$
We also denote $H^{-m;-n}:=(H^{m;n})^*$ the dual space to $H^{m;n}$.

Next, we need to construct the stochastic forcing based on an orthogonal basis of $H$, therefore for $k=(k_1,k_2),$ denote
  \begin{eqnarray*}
  e_k^0&=&(\frac{k_2}{|k|},-\frac{k_1}{|k|})^T\cdot \cos(k\cdot x),
  \\ e_k^1&=&
   (-\frac{k_2}{|k|},\frac{k_1}{|k|})^T\cdot \sin(k\cdot x).
  \end{eqnarray*}
  It is commonsense that $\{e_k^m\}_{k\in \mZ^2\backslash \{0,0\},m\in \{0,1\}}$   forms  an  orthogonal basis of $H_1$ exactly.

Denote
\begin{eqnarray*}
\psi^0_k(x):=(e_k^0,0)^T\in H_1\times H_2,~~ \quad \psi_k^1(x)=(e_k^1,0)^T\in H_1\times H_2,
\end{eqnarray*}
and
\begin{eqnarray}\label{p-71}
\sigma^0_k(x):=(0,e_k^0)^T\in H_1\times H_2, ~~\quad \sigma_k^1(x)=(0,e_k^1)^T\in H_1\times H_2.
\end{eqnarray}

Let $\{e_k^{m}\}_{k\in\mZ_0, m\in\{0,1\}}$ be the standard basis of $\mR^{2|Z_0|}.$ We define a linear map $\mathcal{Q}_b:\mR^{2|Z_0|}\rightarrow H$ such that
\begin{eqnarray*}
  \mathcal{Q}_b e_k^m:=\alpha_k^m \sigma_k^m.
\end{eqnarray*}
Denote the Hilbert-Schmidt norm of $\mathcal{Q}_b$ by
\begin{eqnarray*}
  \cE_0:=\|\mathcal{Q}_b^*\mathcal{Q}_b\|=\sum_{k\in Z_0,m \in \{0,1\}}  (\alpha_k^{m})^2.
\end{eqnarray*}
We consider stochastic forcing of the form
\begin{eqnarray}\label{p-47}
  \mathcal{Q}_b
  \dif W=\sum_{k\in  Z_0,m\in \{0,1\}}\alpha_{k}^{m} \sigma_k^{m}\dif W^{k,m}.
\end{eqnarray}

For  $U=(u,b)^T$ and $\tilde{U}=(\tilde{u},\tilde{b})^T$, denote  $A^{\alpha,\beta}U=((-\Delta )^{\alpha} u, (-\Delta )^{\beta} b )^T$, and
\begin{eqnarray*}
  B(U,\tilde{U})&=& \left(\begin{split}
    &  \Pi \big[ u\cdot \nabla \tilde{u} -b\cdot \nabla \tilde{b}\big]
    \\
    & \Pi \big[ u\cdot \nabla \tilde{b}- b\cdot  \nabla \ti{u}\big]
  \end{split}
  \right),
  \\  B(U)&=&B(U,U),
  \\ F(U)&=&-A^{\alpha,\beta}U-B(U,U).
\end{eqnarray*}
With these preliminaries in hand, the equations (\ref{p-1}) may be written as an abstract stochastic evolution equation on $H$
\begin{eqnarray}\label{p-23}
  \dif U+\big(A^{\alpha,\beta}U+B(U,U)\big)\dif t=\mathcal{Q}_b\dif W, ~~U_0=(u_0,b_0),
\end{eqnarray}
or in a more compact formulation
\begin{eqnarray}\label{pp-23}
  \dif U=F(U)\dif t+\mathcal{Q}_b\dif W.
\end{eqnarray}
We  say that $U = U(t, U_0)$ is a solution of (2.14) if it is $\cF_t$-adapted,
$U \in  C([0, \infty); H) \cap  L^2_{loc}([0, \infty); H^1)$
a.s.
and $U$ satisfies (\ref{p-23}) in the mild sense, that is,
$$
U_t = e^{-A^{\alpha,\beta}t}U_0
-\int_0^t e^{-A^{\alpha,\beta}(t-s)}B(U_s,U_s)\dif s
 +\int_0^t
e^{-A^{\alpha,\beta}(t-s)}G\dif W_s.
$$
The well-posedness   can be established similarly as in \cite{HS}. Hence we let $U=U(t,U_0)$ be the unique solution of  (\ref{p-23}) with initial value $U_0$. For any $\xi=(\xi_1,\xi_2)\in H,$ $t\geq s\geq 0$, the Jacobian $J_{s,t}\xi$ is actually the unique solution of
\begin{eqnarray}\label{p-37}
\left\{
\begin{split}
  & \partial_t  J_{s,t}\xi+ A^{\alpha,\beta}J_{s,t}\xi +B(U_t,J_{s,t}\xi)+B(J_{s,t}\xi,U_t)=0,
  \\ & J_{s,s}\xi=\xi.
  \end{split}
  \right.
\end{eqnarray}
In the interest of brevity, set $J_t\xi:=J_{0,t}\xi.$
Let $J_{s,t}^{(2)}:H\rightarrow \cL(H,\cL(H))$
be the second derivative of $U$ with respect to an initial value  $U_0$. Observe
that for fixed $U_0\in H$ and any $\xi,\xi'\in H$ the function $\varrho=\varrho_t:=J_{s,t}^{(2)}(\xi,\xi')$ is the solution of
\begin{eqnarray}\label{J2}
   \partial_t \varrho_t +A^{\alpha,\beta}\varrho_t +\nabla B(U_t)\varrho_t+\nabla B(J_{s,t}\xi)J_{s,t}\xi'=0,~~~~\quad \varrho_s=0,
\end{eqnarray}
where $\nabla B(\theta)\vartheta=B(\theta,\vartheta)+B(\vartheta,\theta).$

Let  $d=2| Z_0|$.
The Malliavin derivative $\cD:L^2(\Omega,H)\rightarrow L^2(\Omega; L^2(0,T;\mR^{d})\times H)$ satisfies, for each $v\in L^2(0,T;\mR^{d}) $
\begin{eqnarray*}
  \langle \cD U,v\rangle_{ L^2(0,T;\mR^{d}) }=\lim_{\eps \rightarrow 0}\frac{1}{\eps}
  \Big(U(T,U_0,W+\eps \int_0^\cdot v_s \dif s)-U(T,U_0,W)\Big).
\end{eqnarray*}
One may infer from Duhamel's formula that (c.f. \cite{Hairer}) for $v\in L^2(\Omega; L^2(0,T;\mR^{d})) ,$
 \begin{eqnarray*}
  \langle \cD U,v\rangle_{ L^2(0,T;\mR^{d}) }=\int_0^TJ_{s,T}\mathcal{Q}_b v_s \dif s.
\end{eqnarray*}

We define the random operator $\cA_{s,t}:L^2(s,t;\mR^{d})\rightarrow H$ by
\begin{eqnarray*}
\cA_{s,t}v:=\int_{s}^{t}J_{r,t}\mathcal{Q}_b v_r\dif r.
\end{eqnarray*}
Direct computation shows that $\cA_{s,t}v$ satisfies the following equation
\begin{eqnarray*}
  \left\{
\begin{split}
  & \partial_t  \cA_{s,t}v + A^{\alpha,\beta}\cA_{s,t}v +B(U_t,\cA_{s,t}v)+B(\cA_{s,t}v,U_t)=\mathcal{Q}_b v_t,
  \\ & \cA_{s,s}v=0.
  \end{split}
  \right.
\end{eqnarray*}
For any $s<t$, let $\cA_{s,t}^*:H\rightarrow L^2(s,t;\mR^{d})$  be the adjoint of $\cA_{s,t}$, then
\begin{eqnarray*}
(\cA_{s,t}^*\xi)(r)=\mathcal{Q}_b^{*}\cK_{r,t}\xi,  \text{ for any }\xi\in H, r\in [s,t],
\end{eqnarray*}
where $\mathcal{Q}_b^*:H\rightarrow \mR^{d}$ is the adjoint of $\mathcal{Q}_b$, and for $s<t,~\cK_{s,t}\xi$ is the solution of the following "backward" system
\begin{eqnarray}\label{p-3}
 \partial_s \rho^* =A^{\alpha,\beta}\rho^*+(\nabla B(U))^{*}\rho^*=-(\nabla F(U))^*\rho^*,\quad \rho^*_t=\xi.
\end{eqnarray}

It is time to define the Malliavin matrix as
\begin{eqnarray*}
  \cM_{s,t}:=\cA_{s,t}\cA_{s,t}^*:H\rightarrow H.
\end{eqnarray*}
Observe that $\rho_t:=J_{0,t}\xi-\cA_{0,t}v$ satisfies
\begin{eqnarray*}
\left\{
\begin{split}
  & \partial_t  \rho_t+ A^{\alpha,\beta}\rho_t +B(U_t,\rho_t)+B(\rho_t,U_t)=-\mathcal{Q}_b v_t,
  \\ & \rho_0=\xi.
  \end{split}
  \right.
\end{eqnarray*}
This equation enables us to translate the ergodicity issue into a control problem. Actually in conjunction with the Malliavin integration by parts formula, one can obtain the estimate on $\nabla P_t\Phi$ through  spectral analysis on the Malliavin matrix $\mathcal{M}$ (c.f. Section 5).

\subsection{Main theorem}

Before stating the main theorem of the manuscript, let us recall some basic notations with regard to the associated Markovian semigroup. It is necessary to introduce new functional spaces first.

Denote by  $M_b(H)$ and $C_b(H)$ respectively, the spaces of bounded measurable and bounded continuous real valued functions on $H$ equipped with the supremum norm.
We also define
\begin{eqnarray*}
  \cO_\eta:&=&\left\{ \Phi \in C^1(H):\|\Phi\|_\eta<\infty\right\},
  \\  &&\text{where } \|\Phi \|_\eta:=\sup_{U_0\in H}\left(\exp{(-\eta\|U_0\|})(|\Phi(U_0)|+\|\nabla \Phi(U_0)\|)\right),
\end{eqnarray*}
for any $\eta>0,$  which is the special admissible functional space for Theorem \ref{p-27}.

The transition function associated to (\ref{p-23}) is given by
\begin{eqnarray*}
  P_t(U_0,E)=\mP(U(t,U_0)\in E) \text{ for any } U_0 \in H, E\in \cB(H), t\geq 0,
\end{eqnarray*}
where $\cB(H)$ is the collection of Borel sets on $H$,  $U(t, U_0)$ is  the  solution of  (\ref{p-23}) with initial value $U_0$.
We also define the Markov semigroup $\{P_t\}_{t\geq 0}$ with $P_t:M_b(H)\rightarrow M_b(H)$ associated to (\ref{p-1}) by
\begin{eqnarray}\label{p-26}
  P_t \Phi (U_0) :=\mE \Phi(U(t,U_0))=\int_H \Phi(U)P_t(U_0,\dif U) \text{ for any } \Phi \in M_b(H), t\geq 0.
\end{eqnarray}

Now we will give our main results in this article.
\begin{thm}\label{p-27}
Assume Hypothesis  \ref{p-21} holds, then there exists an unique invariant measure $\mu_*$ associated to (\ref{p-1}) and for each $t\geq 0$ the map $P_t$ is ergodic relative to  $\mu_*$. Moreover, there exists a constant $\eta^*$ such that $\mu_*$ satisfies for each $\eta\in (0,\eta^*)$
\begin{itemize}
  \item[({\romannumeral1}) ] (Mixing) There is  $\gamma=\gamma(\eta)>0$ and $C=C(\eta)$ such that
  \begin{eqnarray}\label{p-31}
    \left| \mE \Phi(U(t,U_0))-\int_H \Phi(\bar{U})\dif \mu_*(\bar{U}) \right|
    \leq C\exp{(-\gamma t+\eta \|U_0\|)} \| \Phi \|_\eta
  \end{eqnarray}
  holds for any  $\Phi \in \cO_\eta,U_0\in H$ and any $t\geq 0.$
   \item[({\romannumeral2}) ] (Weak law of large numbers)
   For any $\Phi\in \cO_\eta$ and any $U_0\in H$,
   \begin{eqnarray}\label{p-35}
     \lim_{T\rightarrow \infty}\frac{1}{T}\int_0^T \Phi(U(t,U_0))\dif t=\int_H \Phi(\bar{U})\dif \mu_*(\bar{U})=:m_{\Phi}, \text{  in probability.}
   \end{eqnarray}
  \item[({\romannumeral3}) ] (Central limit theorem) For any $\Phi\in \cO_\eta$, every $U_0\in H$ and  $\xi\in \mR$
      \begin{eqnarray}\label{p-36}
        \lim_{T\rightarrow \infty}\mP\left(\frac{1}{\sqrt{T}}\int_0^T (\Phi(U(t,U_0))-m_{\Phi})\dif t<\xi\right)=\cX(\xi),
      \end{eqnarray}
      where $\cX$ is the distribution function of a normal random variable with zero mean and variance equal to
      \begin{eqnarray*}
          \lim_{T\rightarrow \infty}\frac{1}{T}\mE \left(\int_0^T (\Phi(U(t,U_0))-m_{\Phi}) \dif t\right)^2.
      \end{eqnarray*}
\end{itemize}
\end{thm}

\begin{Rem}\label{rem hypo}
Interestingly, under Hypothesis \ref{p-21} it is possible that the noise allows to be so degenerate that only four modes in the magnetic direction are actually driven. The next example allows one to get a primary idea into this phenomenon, and meanwhile to notice the specificity of Hypothesis \ref{p-21} in comparison to \cite{FGRT,martin,Hairer02}.
\end{Rem}

\begin{exmp}\label{exmp hypo}
If  $ Z_0=\{ (0,1), (1,1), (1,0), (1,2)\}$, then Hypothesis \ref{p-21} holds.
\end{exmp}
\begin{proof}
 For $n\geq 0$, define
   \begin{eqnarray*}
   \hat{Z}_0&=&\{ (0,1), (1,1), -(0,1),-(1,1)\},
   \\
 \hat{Z}_n:&=&\{ k+ \ell  ~\big|~ k\in \hat{Z}_{n-1}, \ell \in \hat{Z}_0, \langle k,\ell^{\bot}\rangle\neq0, |k|\neq |\ell| \},
  \end{eqnarray*}
    then it is  not difficult to check  that
  \begin{eqnarray*}
    \hat{Z}_1=\{(-1,0), (-1,-2), (1,0),(1,2)\},
  \end{eqnarray*}
  and
  \begin{eqnarray*}
    \cup_{n=0}^\infty \hat{Z}_n= \mZ^2 \backslash \{0,0\}.
  \end{eqnarray*}

  By  (\ref{p-22}),
  \begin{eqnarray*}
   \cZ_0 &=&  \hat{Z}_0 \cup \hat{Z}_1,
   \\  \cZ_1 &\supseteq  &   \hat{Z}_1 \cup \hat{Z}_2,
   \\   \cZ_2 &\supseteq  &    \hat{Z}_2 \cup \hat{Z}_3,
   \\   &\cdots &
   \\   \cZ_k &\supseteq &  \hat{Z}_{k} \cup \hat{Z}_{k+1},
      \\   &\cdots &.
  \end{eqnarray*}
Therefore, one sees that
\begin{eqnarray*}
     \cup_{k=0}^{\infty}\cZ_{2k}\supseteq \cup_{n=0}^{\infty}\hat{Z}_{n} = \mZ^2 \backslash \{0,0\},\quad\quad
   \cup_{k=0}^{\infty}\cZ_{2k+1}\supseteq \cup_{n=0}^{\infty}\hat{Z}_{n} = \mZ^2 \backslash \{0,0\},
   \end{eqnarray*}
which yields the desired result.
\end{proof}

\section{Moment  estimates on $U_t,J_{s,t}\xi, \cK_{s,t}\xi, J_{s,t}^{(2)}(\xi,\xi'),\cM_{s,t}$.  }

In this section we provide moment bounds with respect to the unique solution $U$ and its linearizations. They may seem familiar for readers who are familiar with research works with regard to ergodicity on the stochastic Navier-Stokes equations and so on. Hence some proofs are sketched or omitted if they do not distinguish from existing methods. However for the fractional MHD equations \eqref{p-1}, we have to impose $\alpha>1,\beta>1$ to compensate for the complicated advective operator $B$. This is accomplished through delicate interpolation and weighting. Lemma \ref {p-30} and Lemma \ref {p-29} give one a glimpse of this strategy.
\begin{lemma}\label{p-30}
  For any $U_0\in H$, let $U_t=U(t,U_0)$ be the unique solution of (\ref{p-23})
  with initial value  $U_0$. Then there exists $\eta^*>0$ such that
  \begin{enumerate}
    \item  for   any $\eta\in (0,\eta^*]$ and   some $C=C(\eta,\cE_0)>0$, there holds
  \begin{eqnarray*}
    \mE\[\exp\Big\{\eta \|U_t\|^2+\frac{\eta}{2}e^{- t/2}\int_0^t \|\Lambda^\alpha u_s\|^2\dif s +\frac{\eta}{2}e^{- t/2} \int_0^t
    \|\Lambda^\beta b_s\|^2\dif s   \Big\}\]\leq C\exp\{\eta \|U_0\|^2e^{-  t}\}.
  \end{eqnarray*}

    \item  for some $C>0$ and   any $\eta\in (0,\eta^*],$
    \begin{eqnarray}\label{p-33}
  \mE  \exp\Big\{\eta \|U_t\|^2-\eta \|U_0\|^2+\eta \int_0^t \|\Lambda^\alpha u_s\|^2\dif s +\eta \int_0^t \|\Lambda^\beta b_s\|^2\dif s- \eta \cE_0t\Big\}\leq C.
    \end{eqnarray}
    \item For any  $N>0$ and $\eta\in (0,\eta^*]$,
    \begin{eqnarray*}
      \mE \exp\left\{\eta \sum_{k=0}^N\|U_k\|^2 \right\}\leq \exp{(\rho \eta \|U_0\|^2+\kappa N)},
    \end{eqnarray*}
    where $\rho,\kappa>0$ are positive constants independent of $N$ and $U_0.$
    \item  For any $s\geq 0, p\geq 2$, and $\eta\in(0,\eta^*]$, there exists $C=C(\eta,s,T,p)$ such that
        \begin{eqnarray*}
          \mE \left(\sup_{t\in [T/2,T]}\|U_t\|^p_{H^s}\right) &\leq &  C\exp{(\eta\|U_0\|^2)},
          \end{eqnarray*}
          and
          \begin{eqnarray*}
          \mE\left( \|U\|^p_{C^{1/4}([T/2,T],H^s)}\right) &\leq &   C\exp{(\eta\|U_0\|^2)}.
        \end{eqnarray*}
  \end{enumerate}
\end{lemma}
\begin{proof}
(1) By Ito's formula, for $\eta>0,$
  \begin{eqnarray*}
   \eta \|U_t\|^2-\eta \|U_0\|^2+2\eta \int_0^t \|\Lambda^\alpha u_s\|^2\dif s +2\eta \int_0^t \|\Lambda^\beta b_s\|^2\dif s &=& \eta \cE_0t +2\eta \int_0^t \langle b_s, \mathcal{Q}_b\dif W_s\rangle.
  \end{eqnarray*}
  Set  $\bar{Z}_t:=\eta  \|\Lambda^\alpha u_s\|^2+ \eta  \|\Lambda^\beta b_s\|^2$,
  then
  \begin{eqnarray*}
    \eta \cE_0-2\eta \|\Lambda^\alpha u_s\|^2-2 \eta  \|\Lambda^\beta b_s\|^2
    &\leq  &   \eta \cE_0 -2\bar{Z}_t,
    \\ 4\eta^2 |\langle b,\mathcal{Q}_b\rangle|^2 &\leq &  4\eta \cE_0  \bar{Z}_t.
  \end{eqnarray*}
 Applying \cite[lemma 5.1]{Hairer02} with $\bar{U}_t:=\eta \|U_t\|^2$,  one arrives  that there exists $\eta^*>0,$ such that for any $\eta\in (0,\eta^*]$
  \begin{eqnarray*}
    \mE\[\exp\Big\{\eta \|U_t\|^2+\frac{1}{2}e^{-t/2}\int_0^t \eta \|\Lambda^\alpha u_s\|^2\dif s +\frac{1}{2}e^{- t/2}\int_0^t \eta \|\Lambda^\beta b_s\|^2\dif s   \Big\}\]\leq C(\eta,\cE_0)\exp\{\eta \|U_0\|^2e^{- t}\}.
  \end{eqnarray*}

 (2)  Ito's formula yields that
  \begin{eqnarray*}
   \|U_t\|^2-\|U_0\|^2+2\int_0^t \|\Lambda^\alpha u_s\|^2\dif s +2\int_0^t \|\Lambda^\beta b_s\|^2\dif s &=& \cE_0t +2 \int_0^t \langle b_s, \mathcal{Q}_b\dif W_s\rangle,
  \end{eqnarray*}
  then  for any $\eta>0, $
      \begin{eqnarray*}
   \eta \|U_t\|^2-\eta \|U_0\|^2+2\eta \int_0^t \|\Lambda^\alpha u_s\|^2\dif s +\eta \int_0^t \|\Lambda^\beta b_s\|^2\dif s- \eta \cE_0t  &\leq &2 \eta \int_0^t \langle b_s, \mathcal{Q}_b\dif W_s\rangle-\eta \int_0^t \|\Lambda^\beta b_s\|^2\dif s .
  \end{eqnarray*}
    If $\eta\leq  \frac{1}{4 |\cE_0|^2}, $  the following inequality holds from the exponential martingale argument for some absolute constant $C,$
    \begin{eqnarray*}
  \mE  \exp\Big\{\eta \|U_t\|^2-\eta \|U_0\|^2+\eta \int_0^t \|\Lambda^\alpha u_s\|^2\dif s +\eta \int_0^t \|\Lambda^\beta b_s\|^2\dif s- \eta \cE_0t\Big\}\leq C.
    \end{eqnarray*}

    (3) The proof of (3) follows similarly as in \cite[proof of Lemma 4.10]{martin}.

    (4)  The proof of (4) follows similarly as in \cite[Proposition 2.4.12]{KS2012}
    and the fact $\|W^{k,\ell}\|_{C^{1/4}_{[T/2,T]}}$ has finite $p$th moment for any $p\geq 1.$
\end{proof}

The next lemmata include necessary estimates on linearizations of \eqref{p-23}. Referring back to \eqref{p-37}, \eqref{p-3} and \eqref{J2}, one finds that $J_{s,t}$ is the Jacobian operator with its adjoint $\cK_{s,t}$ and derivative $J^{(2)}_{s,t}$. At first glance the following bounds are closely related to those on the Malliavin derivative $\mathcal{M}_{s,t}$. The technically oriented readers may jump to Section 5 for further details.
\begin{lemma}\label{p-29}
 For $\xi\in H,$ assume $J_{s,t}\xi=(J_{s,t}^1\xi,J_{s,t}^2\xi)\in H_1\times H_2.$
For each  $\eta>0$ and $0<s<t,$ we have the following  estimate
\begin{eqnarray}\label{p-38}
\nonumber  &&  \|J_{s,t}\xi\|^2+\int_s^t \big[ \|\Lambda^\alpha J_{s,r}^1\xi \|^2+\|\Lambda^\beta  J_{s,r}^2\xi \|^2 \big] \dif r
  \\ && \leq C \exp{\left(\eta \int_s^t \|U_s\|^2_{H^1}\dif s+C(\eta)(t-s)\right)}\|\xi\|^2,
\end{eqnarray}
where $C,C(\eta)$ is independent of $s,t.$  Moreover, for each $\tau\leq T,p\geq 1$ and any  $\eta>0 $, there exists $C=C(\eta,T-\tau,p)$ such that
\begin{eqnarray}
  \label{p-11}\mE \sup_{s<t \in [\tau,T]}\|J_{s,t}\xi\|^p &\leq &  C\exp{(\eta \|U_0\|^2)} \|\xi\|^p,
   \\  \label{p-12} \mE \sup_{s<t \in [\tau,T]}\|\cK_{s,t}\xi\|^p &\leq &  C\exp{(\eta \|U_0\|^2)}\|\xi\|^p,
  \\  \label{p-13}\mE \sup_{s<t \in [\tau,T]}\|J_{s,t}^{(2)}(\xi,\xi')\|^p &\leq &  C\exp{(\eta \|U_0\|^2)}\|\xi\|^p\|\xi'\|^p.
\end{eqnarray}
\end{lemma}
\begin{proof}
Recalling (\ref{p-37}), for $\alpha'\in (1,\alpha),\beta'\in (0,\beta)$ and any $\eta\in (0,1)$, we deduce from the interpolation inequality and Young inequality that
\begin{eqnarray*}
 \dif    \|J_{s,t}\xi\|^2& =& - 2 \langle A^{\alpha,\beta}J_{s,t}\xi, J_{s,t}\xi\rangle \dif t -2\langle B(U_t,J_{s,t}\xi), J_{s,t}\xi \rangle \dif t - 2 \langle B(J_{s,t}\xi,U_t), J_{s,t}\xi\rangle\dif t
 \\ &=& - 2 \langle A^{\alpha,\beta}J_{s,t}\xi, J_{s,t}\xi\rangle \dif t - 2 \langle B(J_{s,t}\xi,U_t), J_{s,t}\xi\rangle\dif t
 \\ &\leq & -2 \|\Lambda^\alpha J_{s,t}^1\xi \|^2-2 \|\Lambda^\beta  J_{s,t}^2\xi \|^2
 +C\|U_t\|_{H^1}\cdot\[ \|\Lambda^{\alpha'}  J_{s,t}^1\xi \|+ \|\Lambda^{\beta'}  J_{s,t}^2\xi \| \]\cdot \|J_{s,t}\xi  \|
 \\ &\leq & -2 \|\Lambda^\alpha J_{s,t}^1\xi \|^2-2 \|\Lambda^\beta  J_{s,t}^2\xi \|^2
 +C(\eta) \cdot\[ \|\Lambda^{\alpha'}  J_{s,t}^1\xi \|^2+ \|\Lambda^{\beta'}  J_{s,t}^2\xi \|^2  \] +\eta \|U_t\|_{H^1}^2\cdot \|J_{s,t}\xi  \|^2
 \\ &\leq &   - \|\Lambda^\alpha J_{s,t}^1\xi \|^2-\|\Lambda^\beta  J_{s,t}^2\xi \|^2 + \eta \|U_t\|_{H^1}^2\cdot \|J_{s,t}\xi  \|^2+C(\eta) \|J_{s,t}\xi\|^2,
\end{eqnarray*}
which leads to (\ref{p-38}).

(\ref{p-11}) and (\ref{p-12}) follows from  (\ref{p-38}) and Lemma \ref{p-30}.

For fixed $U_0\in H$ and any $\xi,\xi'\in H$, it follows from (\ref{p-37}) that the second derivative $\varrho_t:=J_{s,t}^{(2)}(\xi,\xi')=(\varrho_t^1,\varrho_t^2)\in H_1\times H_2$ satisfies
\begin{eqnarray*}
  \partial_t \varrho_t +A^{\alpha,\beta}\varrho_t +\nabla B(U_t)\varrho_t+\nabla B(J_{s,t}\xi)J_{s,t}\xi'=0,~~~~\quad \varrho_s=0.
\end{eqnarray*}
Then
\begin{eqnarray*}
  \partial_t\|\varrho_t\|^2+2\langle A^{\alpha,\beta}\varrho_t,\varrho_t \rangle +2\langle B(\varrho_t,U_t),\varrho_t\rangle+\langle B(J_{s,t}\xi,J_{s,t}\xi'),\varrho_t \rangle
  +\langle B(J_{s,t}\xi',J_{s,t}\xi),\varrho_t \rangle=0.
\end{eqnarray*}
Therefore again by Young inequality and Interpolation inequality, for  any $\alpha'\in (1,\alpha),\beta'\in (0,\beta)$ and  $\eta>0,$
\begin{eqnarray*}
 &&  \partial_t \|\varrho_t\|^2+2 \|\Lambda^\alpha \varrho_t^1\|^2+2 \|\Lambda^\beta \varrho_t^2\|^2
  \\
  & & \leq   C \[\|\Lambda^{\alpha'} \varrho_t^1\|+\|\Lambda^{\beta'} \varrho_t^2\|\]\[\|\varrho_t\|\cdot \|U_t \|_{H^1}\]
  \\&& \quad +C \[\|\Lambda^{\alpha'} \varrho_t^1\|+\|\Lambda^{\beta'} \varrho_t^2 \|\] \[ \|J_{s,t}\xi'\|\cdot \big[ \|\Lambda^{\alpha'}  J_{s,t}^1\xi \|+ \|\Lambda^{\beta'}  J_{s,t}^2\xi \|  \big]  \]
  \\&& \quad +C \[\|\Lambda^{\alpha'} \varrho_t^1\|+\|\Lambda^{\beta'} \varrho_t^2 \|\] \[ \big[ \|\Lambda^{\alpha'}  J_{s,t}^1\xi' \|+ \|\Lambda^{\beta'}  J_{s,t}^2\xi' \|  \big] \cdot   \|J_{s,t}\xi\| \]
  \\ && \leq  C(\eta) \[\|\Lambda^{\alpha'} \varrho_t^1\|^2+\|\Lambda^{\beta'} \varrho_t^2\|^2\] +\eta  \|\varrho_t\|^2\cdot \|U_t \|_{H^1}^2+\eta   \|J_{s,t}\xi'\|^2+\eta  \|J_{s,t}\xi\|^2
  \\ &&\quad +\eta \[ \|\Lambda^{\alpha'} J_{s,t}^1\xi \|^2+\|\Lambda^{\beta'}  J_{s,t}^2\xi \|^2 \] +\eta \[ \|\Lambda^{\alpha'} J_{s,t}^1\xi' \|^2+\|\Lambda^{\beta'}  J_{s,t}^2\xi' \|^2 \]
    \\ && \leq  \eta \[\|\Lambda^{\alpha} \varrho_t^1\|^2+\|\Lambda^{\beta} \rho_t^2\|^2\]+C(\eta)\|\varrho_t\|^2 +\eta  \|\varrho_t\|^2\cdot \|U_t \|_{H^1}^2+\eta   \|J_{s,t}\xi'\|^2+\eta  \|J_{s,t}\xi\|^2
  \\ &&\quad +\eta \[ \|\Lambda^{\alpha'} J_{s,t}^1\xi \|^2+\|\Lambda^{\beta'}  J_{s,t}^2\xi \|^2 \] +\eta \[ \|\Lambda^{\alpha'} J_{s,t}^1\xi' \|^2+\|\Lambda^{\beta'}  J_{s,t}^2\xi' \|^2 \].
\end{eqnarray*}
Setting $\eta$ small enough,
 one reaches from Gronwall's inequality that
\begin{eqnarray*}
&& \|\varrho_t\|^2
\\ && \leq C \eta   \int_s^t  \big[ \|J_{s,r}\xi'\|^2+ \|\Lambda^{\alpha'} J_{s,r}^1\xi' \|^2+\|\Lambda^{\beta'}  J_{s,r}^2\xi' \|^2 \big]  \dif r \cdot \exp{(\eta \int_s^t  \|U_r \|_{H^1}^2\dif r+ C(\eta)(t-s))}
\\ &&~~+
  C \eta   \int_s^t  \big[ \|J_{s,r}\xi\|^2+ \|\Lambda^{\alpha'} J_{s,r}^1\xi \|^2+\|\Lambda^{\beta'}  J_{s,r}^2\xi \|^2 \big]  \dif r \cdot \exp{(\eta \int_s^t  \|U_r \|_{H^1}^2\dif r+ C(\eta)(t-s))}.
\end{eqnarray*}
Now,  Lemma \ref{p-30} in combination with (\ref{p-38})-(\ref{p-12}) leads to (\ref{p-13}).
\end{proof}

The next lemma is frequently used in Section 5.2, despite the estimate involves a weak norm with respect to initial time only.
\begin{lemma}\label{p-10}
 For any $p\geq 2,T\geq 0,$ and $\eta>0$ there exists $C=C(p,T,\eta)$ such that
 \begin{eqnarray*}
   \mE \sup_{t\in [T/2,T]} \|\partial_t K_{t,T}\xi\|^p_{H^{-2\alpha;-2\beta}} \leq C\exp{(\eta \|U_0\|^2)}\|\xi\|^p.
 \end{eqnarray*}
\end{lemma}
\begin{proof}
Since $\rho_t^*=((\rho_t^*)^1,(\rho_t^*)^2)=K_{t,T}\xi$ satisfies the following equation
 \begin{eqnarray*}
 && \partial_t \rho^*_t =  A^{\alpha,\beta}\rho^*_t+(\nabla B(U_t))^{*}\rho^*_t=-(\nabla F(U_t))^*\rho^*_t,\quad \rho^*_T=\xi,
\end{eqnarray*}
it is immediate that
$$  \|A^{\alpha,\beta}\rho^*_t\|_{H^{-2\alpha; -2\beta}} \leq  \|\rho^*_t\|.$$
Notice that
 \begin{eqnarray*}
\|(\nabla B(U_t))^{*}\rho^*_t\|_{H^{-2\alpha;-2\beta}} &\leq & \sup_{\|\psi\|_{H^{2\alpha;2\beta}}\leq 1}|\langle (\nabla B(U_t))^{*}\rho^*_t , \psi\rangle|
\\ &\leq & \sup_{\|\psi\|_{H^{2\alpha;2\beta}}\leq 1}|\langle\rho^*_t ,  (\nabla B(U_t))\psi\rangle|
\\ & \leq &\|\rho^*_t\|\cdot  \|U_t\|_{ H^1 },
\end{eqnarray*}
then this lemma follows from  Lemma \ref{p-29}.
\end{proof}

For any $N\geq 1,$ define
\begin{eqnarray*}
  H_N:=span\{ \sigma_k^0,\sigma_k^1,\psi_k^0,\psi_k^1~:~ 0<|k|\leq N \},
\end{eqnarray*}
along with the associated projection operators
\begin{eqnarray*}
  P_N:H\rightarrow H_N \text{ the orthogonal projection onto $H_N$ }, ~~Q_N:=I-P_N.
\end{eqnarray*}

The following three  lemmas are particularly useful in translating the bounds on the Malliavin matrix into gradient estimates on the Markov semigroup (c.f. Proposition \ref{p-28}). Since their proofs adopt similar approach as above in combination with a straightforward modification of existing methods (c.f. \cite{FGRT,martin}), they are omitted to save space.
\begin{lemma}
  For every $p\geq 1,T>0,\delta,\gamma>0$ there exists  $N_*=N_*(p,T,\delta,\gamma),$ such that  for any $N\geq N_*$ one has
  \begin{eqnarray*}
    \mE \|Q_NJ_{0,T}\|_{\cL(H,H)}^p &\leq & \gamma \exp{(\delta \|U_0\|^2)},
    \\
       \mE \|J_{0,T}Q_N \|_{\cL(H,H)}^p &\leq & \gamma \exp{(\delta \|U_0\|^2)}.
  \end{eqnarray*}
    Here, $\cL(X,Y)$ denotes the operator norm of the linear map between the given Hilbert spaces $X$ and $Y.$
\end{lemma}

\begin{lemma}
  For $0<s<t$,
  \begin{eqnarray*}
    \|\cA_{s,t}\|_{\cL\left(L^2([s,t],\mR^d),H\right)}\leq C\left(\int_s^t \|J_{r,t}\|_{\cL(H,H)}^2\dif r\right)^{1/2}
  \end{eqnarray*}
  holds for a constant $C$ independent of $s,t.$ Moreover, for any $\kappa>0$
  \begin{eqnarray*}
    \|\cA_{s,t}^* (\cM_{s,t}+\kappa I)^{-1/2}\|_{\cL\left(H, L^2([s,t],\mR^d)\right)}& \leq & 1,
    \\
    \| (\cM_{s,t}+\kappa I)^{-1/2}\cA_{s,t}\|_{\cL\left( L^2([s,t],\mR^d),H\right)}& \leq & 1,
        \\
    \| (\cM_{s,t}+\kappa I)^{-1/2}\|_{\cL\left( H,H\right)}& \leq & \kappa^{-1/2}.
  \end{eqnarray*}
\end{lemma}
Recall that $ \cD$ is the Malliavin derivative. We adopt the notions
$$\cD_sF:=(\cD F)(s),\quad s\in[0,T],\quad\quad \cD^jF:=(\cD F)^j,\quad j=1,\ldots,d.$$
Then observe that for $\tau\leq t$
\begin{eqnarray*}
  \cD_\tau^jJ_{s,t}\xi=
  \left\{\begin{split}
    & J_{\tau,t}^{(2)}(\mathcal{Q}_b e_j,J_{s,\tau}\xi) \text{  if } s\leq \tau,
    \\
    & J_{s,t}^{(2)}(J_{\tau,s}\mathcal{Q}_b e_j,\xi) \text{  if } s>\tau.
  \end{split}
  \right.
\end{eqnarray*}

\begin{lemma}
  For any $\eta>0, \xi\in H$ and $p\geq 1$ we have the bounds
  \begin{eqnarray*}
    \mE\|\cD_\tau^jJ_{s,t}\xi\|^p &\leq & C\exp{(\eta \|U_0\|^2)}\|\xi\|^p,
    \\    \mE\|\cD_\tau^j\cA_{s,t}\|^p_{\cL\left( L^2([s,t],\mR^d),H\right)} &\leq & C\exp{(\eta \|U_0\|^2)},
      \\    \mE\|\cD_\tau^j\cA_{s,t}^*\|^p_{\cL\left( H, L^2([s,t],\mR^d)\right)} &\leq & C\exp{(\eta \|U_0\|^2)},
  \end{eqnarray*}
  where $C=C(\eta, p)$.
\end{lemma}

\section{Details of  Lie bracket computations}

For any Fr\'echet differentiable $E_1,E_2:H\rightarrow H,$
\begin{eqnarray*}
  [E_1,E_2](u):=\nabla E_2(u)E_1(u)-\nabla E_1(u) E_2(u).
\end{eqnarray*}
$[E_1,E_2]$ is referred to as the  Lie bracket of two "vector fields"  $E_1,E_2.$



This section is technical, however, reveals some intrinsic thoughts of the manuscript. As a matter of fact, we present that for any $N\in\mathbb{N}$, how finite dimensional subspaces $H_N$ of $H$ can be generated through the iterations of Lie brackets. It is worth mentioning that these computations are motivated by the celebrated H$\mathrm{\ddot{o}}$rmander condition for the Kolmogorov-Fokker-Planck equations associated to \eqref{p-23}. The following is split into two parts. Firstly, we describe how the velocity direction $u$ is covered.

\subsection{Covering velocity direction}

For $u, \tilde{u}\in H_1=H_2$, denote $\mathbf{b}(u,\tilde{u}):=u\cdot  \nabla \tilde{u}.$
For any $\ell,k \in \mZ^{2},  m,m'\in \{0,1\}$ and $U=(u,b)\in H_1\times H_2$,  we introduce
 \begin{eqnarray}
  \nonumber   Y_k^{m}(U): &=& [F(U),\sigma_k^{m}]
\\ \nonumber  &=&  A^{\alpha,\beta}\sigma_k^{m}+B(\sigma_k^{m},U)+B(U,\sigma_k^{m}),
  \\ \nonumber  J_{k,\ell}^{m,m'}(U):&=&  -[ Y_k^{m}(U),\sigma_\ell^{m'}]
   \\ \nonumber  &=& B(\sigma_k^{m},\sigma_\ell^{m'})+B(\sigma_\ell^{m'},\sigma_k^{m})
  \\ \nonumber  &=& \left(
  \begin{split}
    &  -\Pi \mathbf{b}(e_k^m,e_{\ell}^{m'}) -\Pi \mathbf{b}(e_\ell^{m'},e_k^m)
    \\
    & 0
  \end{split}\right)
  \\ \label{p-49}  :&=& \left(
  \begin{split}
    & -\Pi \cJ_{k,\ell}^{m,m'}
    \\
    & 0
  \end{split}\right).
\end{eqnarray}

In fact, $Y_k^{m}(U)$ and $J_{k,\ell}^{m,m'}(U)$ are devised elaborately by calculation to guarantee that the following two lemmas hold.

\begin{lemma}\label{p-19}
For $k,\ell\in  \mZ_{+}^2,$
\begin{eqnarray*}
&&  \mathbf{b}(e_k^1,e_\ell^1)=\frac{\langle k,\ell^{\bot} \rangle  }{|k||\ell|}\sin(k\cdot x)\cos(\ell \cdot x)(\ell_2,-\ell_1)^T,
 \\ &&  \mathbf{b}(e_\ell^1 ,e_k^1)=\frac{\langle \ell,k^{\bot}\rangle  }{|k||\ell|}\sin(\ell \cdot x)\cos(k \cdot x) (k_2,-k_1)^T,
  \\ &&  \mathbf{b}(e_k^1 ,e_{\ell}^0)=\frac{\langle k,\ell^{\bot} \rangle }{|k||\ell|}\sin(k \cdot x)\sin(\ell  \cdot x)(-\ell_2,\ell_1)^T,
    \\ &&  \mathbf{b}(e_\ell^1 ,e_{k}^0)=\frac{\langle \ell,k^{\bot} \rangle }{|k||\ell|}\sin(\ell \cdot x)\sin(k  \cdot x)(-k_2,k_1)^T,
 \end{eqnarray*}
 and
 \begin{eqnarray*}
&&  \mathbf{b}(e_{k}^0,e_\ell^1)=\frac{\langle k,\ell^{\bot} \rangle  }{|k||\ell|}\cos(k\cdot x)\cos(\ell \cdot x)(\ell_2,-\ell_1)^T,
 \\ &&  \mathbf{b}(e_{\ell}^0 ,e_k^1)=\frac{\langle \ell,k^{\bot}\rangle  }{|k||\ell|}\cos(\ell \cdot x)\cos(k \cdot x) (k_2,-k_1)^T,
  \\ &&  \mathbf{b}(e_{k}^0 ,e_{\ell}^0)=\frac{\langle k,\ell^{\bot} \rangle }{|k||\ell|}\cos(k \cdot x)\sin(\ell  \cdot x)(-\ell_2,\ell_1)^T,
    \\ &&  \mathbf{b}(e_{\ell}^0 ,e_{k}^0)=\frac{\langle \ell,k^{\bot} \rangle }{|k||\ell|}\cos(\ell \cdot x)\sin(k  \cdot x)(-k_2,k_1)^T.
 \end{eqnarray*}
\end{lemma}

\begin{lemma}\label{p-48}
Let $a=\frac{\langle k,\ell^{\bot}\rangle}{|k||\ell|}$, then for any $k,\ell \in  \mZ_{+}^2$,
  \begin{eqnarray*}
   \cJ_{k,\ell}^{0,1} &=& \mathbf{b}(e_{k}^0,e_\ell^1)+\mathbf{b}(e_\ell^1,e_{k}^0)
    \\ &=& a\cos((k+\ell)x)(\ell_2-k_2,-\ell_1+k_1)^T
 +a\cos((k-\ell)x)(\ell_2+k_2,-\ell_1-k_1)^T,
      \\   \cJ_{\ell,k}^{0,1} &=& \mathbf{b}(e_{\ell}^0,e_k^1)+\mathbf{b}(e_k^1,e_{\ell}^0)
      \\ &=& -a\cos((k+\ell)x)(k_2-\ell_2,-k_1+\ell_1)^T
 -a\cos((\ell-k)x)(\ell_2+k_2,-\ell_1-k_1)^T,
  \end{eqnarray*}
  and furthermore
  \begin{eqnarray}
 \label{p-70}  \cJ_{k,\ell}^{0,1}+ \cJ_{\ell,k}^{0,1} &=&2a\cos((k+\ell)x)(\ell_2-k_2,-\ell_1+k_1)^T,
 \\ \label{p-69} \Pi\big[ \cJ_{k,\ell}^{0,1}+ \cJ_{\ell,k}^{0,1}\big]&=& ac\frac{1}{|k+\ell|}\cdot (|\ell|^2-|k|^2)e_{k+\ell}^0,\\
     \Pi\big[\cJ_{k,\ell}^{0,1}-\cJ_{\ell,k}^{0,1}\big]&=&ac\frac{-|\ell|^2+|k|^2}{|k-\ell|}\cdot   e_{k-\ell}^0,
     \\     \Pi\big[\cJ_{k,\ell}^{1,1}+\cJ_{\ell,k}^{0,0}\big]&=&ac\frac{|\ell|^2-|k|^2}{|k-\ell|}\cdot   e_{k-\ell}^1,
          \\     \Pi\big[\cJ_{k,\ell}^{1,1}-\cJ_{\ell,k}^{0,0}\big]&=&ac\frac{|\ell|^2-|k|^2}{|k+\ell|}\cdot   e_{k+\ell}^1,
   \end{eqnarray}
  where $c$ is an absolutely non-zero constant independent of $k,\ell$ and may change from line to line.
\end{lemma}
\begin{proof}
Since all of the above can be proved in a similar way by direct calculating, we only give the proof of (\ref{p-69}).

  It is from (\ref{p-70}) that
  \begin{eqnarray*}
   \Pi\big[ \cJ_{k,\ell}^{0,1}+ \cJ_{\ell,k}^{0,1}\big]
  &=& \langle 2a\cos((k+\ell)x)(\ell_2-k_2,-\ell_1+k_1)^T, e_{k+\ell}^0\rangle  e_{k+\ell}^0
  \\ &=& ac\frac{1}{|k+\ell|}\cdot (|\ell|^2-|k|^2)e_{k+\ell}^0,
  \end{eqnarray*}
  where $\langle \cdot, \cdot \rangle$ denotes the inner product on $H_1 $, $c$ is a non-zero constant.
\end{proof}

By lemma \ref{p-48} and (\ref{p-49}), we can generate suitable directions in the $u$ component.
\begin{lemma}\label{p-18}
Let $a=\frac{\langle k,\ell^{\bot}\rangle}{|k||\ell|}$,  then for  some  absolutely non-zero constant $c$ which is independent of  $k,\ell$,  the following inequalities hold.
   \begin{eqnarray*}
 J_{k,\ell}^{0,1}+ J_{\ell,k}^{0,1}&=& ac\frac{1}{|k+\ell|}\cdot (|\ell|^2-|k|^2)\psi_{k+\ell}^0,
   \\
J_{k,\ell}^{0,1}-J_{\ell,k}^{0,1}&=&ac\frac{-|\ell|^2+|k|^2}{|k-\ell|}\cdot   \psi_{k-\ell}^0,
     \\
     J_{k,\ell}^{1,1}+J_{\ell,k}^{0,0}&=&ac\frac{|\ell|^2-|k|^2}{|k-\ell|}\cdot   \psi_{k-\ell}^1,
          \\     J_{k,\ell}^{1,1}-J_{\ell,k}^{0,0}&=&ac\frac{|\ell|^2-|k|^2}{|k+\ell|}\cdot   \psi_{k+\ell}^1.
   \end{eqnarray*}
\end{lemma}

\subsection{Covering magnetic direction}

Likewise, we will need the following notations for the $b$ direction, which are also obtained through the iteration of Lie brackets computation.
\begin{eqnarray}
\nonumber   \cY_k^{m}(U):&=& \big[F(U),\psi_k^{m}\big]
  \\ \nonumber   &=&
  A^{\alpha,\beta}\psi_k^{m}+B(\psi_k^{m},U)+B(U,\psi_k^{m})
  \\ \nonumber
  Z_{k,\ell}^{m,m'}:&=&- \Big[ \cY_k^{m}(U),\sigma_{\ell}^{m'}\Big]
  \\ \nonumber    &=& B(\psi_k^{m},\sigma_{\ell}^{m'})+B(\sigma_{\ell}^{m'},\psi_k^{m})
  \\ \nonumber   &=&
  \left(\begin{split}
    &0
    \\
    &\Pi \big[ \mathbf{b}(e_k^{m},e_{\ell}^{m'})-\mathbf{b}(e_\ell^{m'},e_k^{m})\big]
  \end{split}
  \right)
    \\ \label{p-50}   &=&
  \left(\begin{split}
    &0
    \\
    &\Pi\cZ_{k,\ell}^{m,m'}
  \end{split}
  \right),
\end{eqnarray}
where
\begin{eqnarray*}
\cZ_{k,\ell}^{m,m'}:=\mathbf{b}(e_k^{m},e_{\ell}^{m'})-\mathbf{b}(e_\ell^{m'},e_k^{m}).
\end{eqnarray*}
The following lemma is the counterpart of Lemma \ref{p-48}.
\begin{lemma}\label{p-51}
 Denote  $a=\frac{\langle k,\ell^{\bot}\rangle}{|k||\ell|}$,  then
for   some  absolutely non-zero constant   $c$  which is independent of   $k,\ell$(It may changes from line to line),  the following equalities hold.
\begin{eqnarray*}
 \Pi \cZ_{k,\ell}^{0,1} +\Pi \cZ_{\ell,k}^{0,1}  &=&ac|k-\ell |e_{k-\ell}^0,
\\
\Pi \cZ_{k,\ell}^{0,1} -\Pi \cZ_{\ell,k}^{0,1} &=&ac|k+\ell|e_{k+\ell}^0,
  \\ \Pi \cZ_{k,\ell}^{1,1}+\Pi \cZ_{k,\ell}^{0,0}&=&ac|k-\ell|e_{k-\ell}^1,
 \\  \Pi \cZ_{k,\ell}^{1,1}-\Pi \cZ_{k,\ell}^{0,0}&=&ac|k+\ell|e_{k+\ell}^1.
\end{eqnarray*}
\end{lemma}
\begin{proof}
By the   definition of
$\cZ_{k,\ell}^{m,m'}$, we get
  \begin{eqnarray}
\label{1} \cZ_{k,\ell}^{0,1}+  \cZ_{\ell,k}^{0,1}&=& \big(\mathbf{b}(e_k^{0},e_{\ell}^{1})-\mathbf{b}(e_k^{1},e_\ell^{0})\big)+\big(\mathbf{b}(e_\ell^{0},e_{k}^{1}-\mathbf{b}(e_\ell^{1},e_k^{0}))\big).
  \end{eqnarray}
 By  Lemma \ref{p-19},  it holds that
 \begin{eqnarray}
    \nonumber && \big(\mathbf{b}(e_k^{0},e_{\ell}^{1})-\mathbf{b}(e_k^{1},e_\ell^{0})\big)
    \\ \nonumber && = \frac{\langle k,\ell^{\bot} \rangle  }{|k||\ell|}\cos(k\cdot x)\cos(\ell \cdot x)(\ell_2,-\ell_1)^T-\frac{\langle k,\ell^{\bot} \rangle }{|k||\ell|}\sin(k \cdot x)\sin(\ell  \cdot x)(-\ell_2,\ell_1)^T
    \\ \label{2}&&= \frac{\langle k,\ell^{\bot} \rangle  }{|k||\ell|} (\ell_2,-\ell_1)^T \cos (k-\ell)x.
  \end{eqnarray}
  With a similar way,  one sees that
  \begin{eqnarray}
   \label{3} && \big(\mathbf{b}(e_\ell^{0},e_{k}^{1})-\mathbf{b}(e_\ell^{1},e_k^{0})\big)
         = \frac{\langle k,\ell^{\bot} \rangle  }{|k||\ell|}(-k_2,k_1)^T \cos (k-\ell)x.
  \end{eqnarray}
  Combining (\ref{3})(\ref{2}) with (\ref{1}), we obtain
 \begin{eqnarray*}
 \cZ_{k,\ell}^{0,1}+  \cZ_{\ell,k}^{0,1}=a(\ell_2-k_2,-\ell_1+k_1)^T \cos (k-\ell)x.
  \end{eqnarray*}

Therefore,
  \begin{eqnarray*}
 \Pi \cZ_{k,\ell}^{0,1} +\Pi \cZ_{\ell,k}^{0,1}
 &=& \langle a\cos\big((k-\ell)x\big) ( \ell_2-k_2,-\ell_1+k_1)^T, e_{k-\ell}^0\rangle e_{k-\ell}^0
 \\ &=&a\Big\langle \cos\big((k-\ell)x\big) ( \ell_2-k_2,-\ell_1+k_1)^T, \cos\big((k-\ell)x\big) ( \frac{k_2-\ell_2}{|k-\ell|},\frac{-k_1+\ell_1}{|k-\ell|})^T\Big\rangle e_{k-\ell}^0
 \\ &=& ac \frac{|k_2-\ell_2|^2+|k_1-\ell_1|^2 }{|k-\ell|}e_{k-\ell}^0=ac|k-\ell |e_{k-\ell}^0 ,
  \end{eqnarray*}
  where
  \begin{eqnarray*}
    c&=& -\int_{[-\pi,\pi]^2} \cos^2\big((k_1-\ell_1)x_1+(k_2-\ell_2)x_2\big) d x_1dx_2
    \\ &=&  -\int_{[-\pi,\pi]^2} \frac{1+\cos 2\big((k_1-\ell_1)x_1+(k_2-\ell_2)x_2\big) }{2}d x_1dx_2
    \\ &=& -\frac{1}{2}(2\pi)^2.
  \end{eqnarray*}
 The proof of other equalities are similar.
\end{proof}

Likewise,  by Lemma \ref{p-51} and (\ref{p-50}) we can generate suitable directions in the $b$ component.
\begin{lemma}\label{h-3}
 Denote  $a=\frac{\langle k,\ell^{\bot}\rangle}{|k||\ell|}$,  then
for   some  absolutely non-zero constant   $c$  which is independent of   $k,\ell$(It may changes from line to line),    the following equalities hold.
\begin{eqnarray*}
Z_{k,\ell}^{0,1} +Z_{\ell,k}^{0,1}  &=&ac|k-\ell |\sigma_{k-\ell}^0,
\\
Z_{k,\ell}^{0,1} -Z_{\ell,k}^{0,1} &=&ac|k+\ell|\sigma_{k+\ell}^0,
  \\ Z_{k,\ell}^{1,1}+Z_{k,\ell}^{0,0}&=&ac|k-\ell|\sigma_{k-\ell}^1,
 \\  Z_{k,\ell}^{1,1}-Z_{k,\ell}^{0,0}&=&ac|k+\ell|\sigma_{k+\ell}^1.
\end{eqnarray*}
\end{lemma}

In conclusion, we give an illustration  in Figure  \ref{t2}   how the new directions generated from the existing directions    via  the iterations  of the chain of  bracket computations. The construction is interesting that in the upper half part $\psi$'s are generated by $\sigma$'s, while in the lower half part $\sigma$'s are generated by $\psi$'s. This antisymmetric relationship is originated from the advective structure in $B$.

\tikzstyle{block} = [rectangle, draw, fill=blue!20, text width=4em, text centered, rounded corners]
\tikzstyle{hugeBlock} = [rectangle, draw, fill=blue!20,
text width=5em, text centered, rounded corners, minimum height=4em]
\tikzstyle{line} = [draw, -latex]

\makeatletter

\def\@captype{figure}
\begin{figure*}[h]
\begin{tikzpicture}[shorten >=5pt,node distance=8cm]
\node(zt) at (0,3){$\sigma_{k}^{m},k\in \cZ_{2n}$};
\node(ft) at (3,3){$Y_k^m(U)$};
\node(et) at (6,3){$J_{k,\ell}^{m,m'}(U)$};
\node(tt) at (9.5,3){$\psi_{k\pm \ell}^m$};
\node(oo) at (-1.5, 1.5){$\sigma_k^m,k\in \cZ_{2n+2}$};
\node(tz) at (9.5,0){$\psi_k^m,k\in\cZ_{2n+1}$};
\node(ez) at (6,0){$\cY_k^m(U)$};
 \node(fz) at (3,0){$Z_{k,\ell}^{m,m'}$};
 \node(zz) at (0,0){$\sigma_{k\pm \ell}^{m}$};
\node() at (7.65,0.25){\tiny $[F(U),\cdot]$  };
\node() at (1.5,3.25){\tiny $[F(U),\cdot]$  };
\node()  at (4.5,3.25){\tiny $-[\cdot, \sigma_\ell^{m'}]$ };
\node()  at (4.5,0.25){\tiny $-[\cdot, \sigma_\ell^{m'}]$ };

\node()  at (7.7,3.25){\tiny  lemma \ref{p-18} };
\node()  at (1.65,0.25){\tiny  lemma \ref{h-3} };

\node() at (-0.75, 2.28){\tiny $n=n+1$  };

\draw[dashed, ->,green](fz) --node{} (zz);

\path[line, ->](zt)-- node{} (ft);
\path[line, ->](ft)-- node{} (et);
\draw[dashed, ->,green](et) --node{} (tt);
      \path[line,->](tz) -- (ez);
         \path[line,->](ez) -- (fz);
   \path[line,double,yellow ](zz) -- (oo);
      \path[line,double, yellow](tt) -- (tz);
   \draw[dotted, ->,red](oo) --node{} (zt);
\end{tikzpicture}
\caption{\footnotesize  An illustration  of   how the new directions generated from the existing directions    via  the iterations  of the chain of  bracket computations.
In this figure, $m,m'\in \{0,1\}, \ell\in \cZ_0.$
Solid arrows mean that the new function is generated from a Lie bracket, with the type  of bracket indicated  above the arrow. Dashed arrows with  green color    signify that the new element is generated as a linear combination of elements from the previous position.
The  dotted   arrows  with  red color     shows that the process is iterative.  The doubled arrow with  yellow color   (\textcolor[rgb]{1.00,1.00,0.00}{$\Rightarrow$})  shows that $k\pm \ell $ is a element belongs to $\cZ_{2n+1}$ or $\cZ_{2n+2}$ actually.
 }
\label{t2}
\end{figure*}
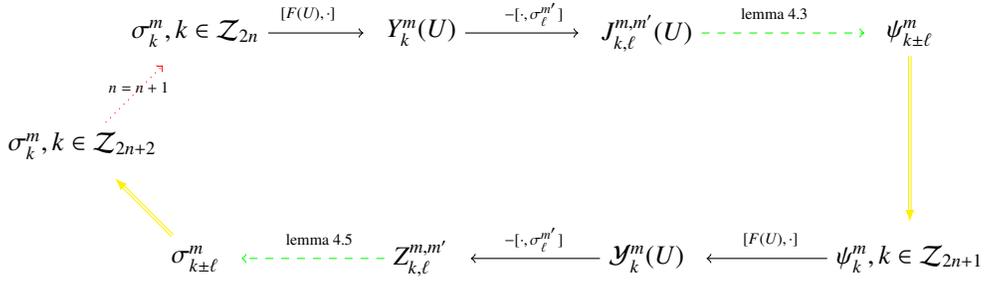

\section{Spectral properties of $\cM$}

For any $\alpha>0, N\in \mN,$ we define
\begin{eqnarray*}
  \cS_{\alpha,N}:=\{\phi\in H:\|P_N\phi\|^2\geq \alpha \|\phi\|^2\}.
\end{eqnarray*}

The aim of this section is to prove the following theorem, which gives information on the probability of eigenvectors with sizable projections in the unstable directions to have small eigenvalues. Broadly speaking, this provides us the invertibility of the Malliavin matrix on the space spanned by the unstable directions. Since it is finite dimensional under current circumstances, one can thus formulate a control problem through the Malliavin integration by parts formula to obtain the gradient estimate on the Markov semigroup, which is extremely useful in establishing ergodicity (c.f. Proposition \ref{p-28}).

\begin{thm}\label{p-9}
For any $N\geq 1,\alpha\in (0,1]$ and $\eta>0,$ there exists a positive constant $\eps^*=\eps^*(\alpha,\eta,N,T)>0,$ such that, for any $n\geq 0,$ and $\eps\in (0,\eps^*]$, there exists a measurable set $\Omega_\eps=\Omega_\eps(\alpha,N,T)\subseteq \Omega$ satisfying
  \begin{eqnarray}\label{p-54}
    \mP(\Omega_\eps^c )\leq r(\eps )\exp{(\eta \|U_0\|^2)},
  \end{eqnarray}
    where $r=r(\alpha,\eta,N,T):(0,\eps^*]\rightarrow (0,\infty)$ is a non-negative, decreasing function with $\lim_{\eps \rightarrow 0}r(\eps)=0,$ and on the set $\Omega_\eps,$
  \begin{eqnarray}\label{p-55}
\inf_{\phi\in \cS_{\alpha,N}}\frac{\langle \cM_{0,T}\phi,\phi\rangle}{\|\phi\|^2}\geq \eps.
  \end{eqnarray}
\end{thm}

In order to prove this theorem, we will first introduce a series of quadratic forms $Q_N$ and their lower bounds, next in Subsection \ref{p-7} we introduce or recall some notational conventions and technical tools which will be used frequently. Then  we estimate upper bounds on $Q_N$ in Subsection \ref{p-8}. Finally in Subsection \ref{p-59}, we complete the proof of Theorem \ref{p-9}. To start with, denote
\begin{eqnarray*}
  \langle Q_N\phi,\phi\rangle:= \sum_{n=0}^N \sum_{k\in \cZ_{2n}, m\in \{0,1\}}|\langle \phi, \sigma_k^{m}\rangle|^2+
 \sum_{n=0}^N \sum_{k\in \cZ_{2n+1}, m\in \{0,1\}}  |\langle \phi, \psi_{k}^{m} \rangle|^2.
\end{eqnarray*}


Lower bounds on these Quadratic forms are fairly simple since we are merely focusing on $\phi \in\cS_{\alpha,N}.$

\begin{proposition}\label{p-53}
  Fix  any integer $N\in \mN$  and $\alpha\in (0,1]$,
  \begin{eqnarray*}
  \langle Q_N\phi,\phi\rangle \geq \frac{\alpha}{2}\|\phi\|^2
  \end{eqnarray*}
   holds  for every $\phi \in\cS_{\alpha,N}.$
\end{proposition}
\begin{proof}
  Its proof  is trivial.
\end{proof}

\subsection{Preliminaries}\label{p-7}

Denote by $\bar{U}=U-\mathcal{Q}_bW,$  then
\begin{eqnarray}\label{p-5}
\left\{
\begin{split}
  & \partial_t \bar{U} =F(U)=F(\bar{U}+\mathcal{Q}_b W),
  \\ &\bar{U}_0=U_0,
  \end{split}
  \right.
\end{eqnarray}
and by expanding $U=\bar{U}+\mathcal{Q}_b W$ we find
\begin{eqnarray}
  Y_k^m(U)=Y_k^m(\bar{U})-\sum_{\ell  \in Z_0,m'\in \{0,1\}}\alpha_\ell^{m'}[Y_k^m(U),\sigma_\ell^{m'}]W^{\ell,m'},
\end{eqnarray}
and
\begin{eqnarray}
  \cY_k^m(U)=\cY_k^m(\bar{U})-\sum_{\ell  \in Z_0,m'\in \{0,1\}}\alpha_\ell^{m'}[\cY_k^m(U),\sigma_\ell^{m'}]W^{\ell,m'}.
\end{eqnarray}

  We introduce for $\alpha \in [0,1],\phi\in H $
\begin{eqnarray*}
  \cN_\alpha(\phi):&=& \max_{\ell \in  Z_0,m' \in\{0,1\}}\Big\{\| \langle \cK_{t,T}\phi,Y_k^m(\bar{U})\rangle\|_{C^\alpha}, \quad  |\alpha_\ell^{m'}|\cdot \| \langle   \cK_{t,T}\phi, [Y_k^m(U),\sigma_{\ell}^{m'}]   \rangle\|_{C^\alpha}\Big\},
  \\
\cM_\alpha(\phi):&=& \max_{\ell \in  Z_0,m' \in\{0,1\}}\Big\{\| \langle \cK_{t,T}\phi,\cY_k^m(\bar{U})\rangle\|_{C^\alpha},  \quad |\alpha_\ell^{m'}|\cdot \| \langle   \cK_{t,T}\phi, [\cY_k^m(U),\psi_\ell^{m'}]   \rangle\|_{C^\alpha}\Big\},
\end{eqnarray*}
where  for any function $g:[T/2,T]\rightarrow \mR$, ~$\|g\|_{C^\alpha}$ is defined by
\begin{eqnarray*}
\|g\|_{C^\alpha}:=\|g\|_{C^\alpha[T/2,T]}:=\sup_{\mbox{\tiny$\begin{array}{c}
t_1\neq t_2\\
t_1,t_2\in [T/2,T]
\end{array}$}}
\frac{|g(t_1)-g(t_2)|}{|t_1-t_2|^{\alpha}},
\end{eqnarray*}
and   for $\alpha=0,$ $\|g\|_{C^0}$ is defined by
\begin{eqnarray*}
\|g\|_{C^0}:=\|g\|_{C^0[T/2,T]}:=\sup_{t\in [T/2,T]}|g(t)|.
\end{eqnarray*}

\begin{lemma}\label{p-14}
For any $p\geq 1,\eta>0,$
\begin{eqnarray}
  \label{p-61} \mE\[\sup_{\phi \in H, \|\phi\|=1 }\cN_0(\phi)^p \] &\leq & C(\eta,k,p)\exp{(\eta \|U_0\|^2)},
  \\ \label{p-62}
   \mE\[\sup_{\phi \in H, \|\phi\|=1 }\cN_1(\phi)^p\] &\leq &  C(\eta,k,p)\exp{(\eta \|U_0\|^2)},
   \\  \label{p-63}  \mE\[\sup_{\phi \in H, \|\phi\|=1 } \cM_0(\phi)^p \] &\leq & C(\eta,k,p)\exp{(\eta \|U_0\|^2)},
  \\ \label{p-64}
   \mE\[\sup_{\phi \in H, \|\phi\|=1 } \cM_1(\phi)^p\] &\leq &  C(\eta,k,p)\exp{(\eta \|U_0\|^2)}.
\end{eqnarray}
\end{lemma}
\begin{proof}
(\ref{p-61}) and (\ref{p-63}) follow directly from    Lemma \ref{p-30} and Lemma \ref{p-29}.

By the expressions  of $Y_k^m(\bar{U}), F(U),Z_{k,\ell}^{m,m'}$ and Lemma \ref{p-30}, there exist $C=C(k,p,\eta)$ and $q=q(p,k)$ such that
\begin{eqnarray}
\nonumber  &&   \mE \sup_{t\in [T/2,T]}\|
 [Y_k^m(\bar{U}), F(U)] \|^{2p} + \mE \sup_{t\in [T/2,T]}\|
[F(U), Z_{k,\ell}^{m,m'} ]  \|^{2p}
 \\ \nonumber  &&  \leq    C\mE \[1+\|U\|_{H^4}^q\]
 \\  \label{p-66}& & \leq  C \exp{(\eta \|U_0\|^2/2)}.
\end{eqnarray}

This along with Lemma  \ref{p-29} yields
\begin{eqnarray}
\nonumber   && \mE  \| \langle \cK_{t,T}\phi,Y_k^m(\bar{U})\rangle\|_{C^1}^p
\\ \nonumber   &&\leq  C \mE  \| \partial_t \langle \cK_{t,T}\phi,Y_k^m(\bar{U})\rangle\|^p
\\ \nonumber   &&\leq  C \mE   | \langle \cK_{t,T}\phi,[Y_k^m(\bar{U}),F(U)]\rangle |^p
 \\  \nonumber  && \leq C\|\phi\|^p \exp{(\eta \|U_0\|^2/2)}\left(\mE \sup_{t\in [T/2,T]}\|
 [Y_k^m(\bar{U}), F(U)] \|^{2p}\right)^{1/2}.
 \\  \label{p-67}  && \leq  C \exp{(\eta \|U_0\|^2 )}.
 \end{eqnarray}

Combining (\ref{p-66}) with Lemma  \ref{p-29}, one arrives that
\begin{eqnarray}
 \nonumber  && \mE\[  \| \langle   \cK_{t,T}\phi, [Y_k^m(U),\sigma_{\ell}^{m'}]   \rangle\|_{C^1}^p  \]
  \\  \nonumber &&=\mE\[  \| \langle   \cK_{t,T}\phi,Z_{k,\ell}^{m,m'}   \rangle\|_{C^1}^p  \]
  \\   \nonumber && \leq \mE\[  | \partial_t \langle   \cK_{t,T}\phi,Z_{k,\ell}^{m,m'}   \rangle |^p  \]
  \\   \nonumber &&
  \leq \mE\[  |  \langle   \cK_{t,T}\phi,[F(U), Z_{k,\ell}^{m,m'} ]  \rangle |^p  \]
  \\  \nonumber  && \leq C\|\phi\|^p \exp{(\eta \|U_0\|^2/2)}\left(\mE \sup_{t\in [T/2,T]}\|
[F(U), Z_{k,\ell}^{m,m'} ]  \|^{2p}\right)^{1/2}
 \\  \label{p-68} && \leq  C \exp{(\eta \|U_0\|^2 )}.
\end{eqnarray}
Immediately (\ref{p-62}) follows from  (\ref{p-67}) and (\ref{p-68}). The proof of
(\ref{p-64}) is similar to that  of   (\ref{p-62}).
\end{proof}

We finish this subsection with citing two technical tools.
\begin{lemma}\label{lemma fght}
(F${\ddot{o}}$ldes et al. \cite{FGRT}) Fix $T>0,~\alpha\in (0,1]$ and an index set $\mathcal{I}$. Consider a collection of random functions $g_\phi$ taking values in $C^{1,\alpha}([T/2,T])$ and indexed by $\phi\in\mathcal{I}$. Define, for each $\eps>0,$
$$\Lambda_{\eps,\alpha}:=\bigcup\limits_{\phi\in\mathcal{I}}\Lambda^\phi_{\eps,\alpha},\quad where~\Lambda^\phi_{\eps,\alpha}:=\left\{\sup\limits_{t\in[T/2,T]}|g_\phi(t)|\leq\eps~and~\sup\limits_{t\in[T/2,T]}|g'_\phi(t)|>\eps^\frac{\alpha}{2(1+\alpha)}\right\}.$$
Then, there is $\eps_0=\eps_0(\alpha,T)$ such that for each $\eps\in (0,\eps_0)$
$$\mathbb{P}(\Lambda_{\eps,\alpha})\leq C\eps\mathbb{E}\left(\sup\limits_{\phi\in\mathcal{I}}\|g_\phi\|^{2/\alpha}_{C^{1,\alpha}[T/2,T]}\right).$$
\end{lemma}

Given any multi-index $\alpha := (\alpha_1, \cdots , \alpha_d) \in N^d $,  recall the standard notation $W^\alpha :=
W_1^{\alpha_1}
\cdots W_d^{\alpha_d}. $
\begin{thm}\label{theorem 6.4}
(Hairer-Mattingly \cite{Hairer}) Fix $M,T>0$. Consider the collection $\mathfrak{B}_M$ of $M$th degree of 'Wiener polynomials' of the form
$$F=A_0+\sum\limits_{|\alpha|\leq M}A_\alpha W^\alpha,$$
where for each multi-index $\alpha$, with $|\alpha|\leq M,$ $A_\alpha:\Omega\times [0,T]\rightarrow\mathbb{R}$ is an arbitrary stochastic process. Then for all $\eps\in(0,1)$ and $\beta>0$, there exists a measurable set $\Omega_{\eps,M,\beta}$ with $\mathbb{P}(\Omega^c_{\eps,M,\beta})\leq C\eps,$ such that on $\Omega_{\eps,M,\beta}$ and for every $F\in\mathfrak{B}_M$
$$\sup\limits_{t\in[0,T]}|F(t)|<\eps^\beta\Rightarrow
\begin{cases}
either & \sup\limits_{\alpha\leq M}\sup\limits_{t\in[0,T]}|A_\alpha(t)|\leq \eps^{\beta3^{-M}},\\
or & \sup\limits_{\alpha\leq M}\sup\limits_{s\neq t\in[0,T]}\frac{A_\alpha(t)-A_\alpha(s)}{t-s}\geq \eps^{-\beta3^{-(M+1)}}.
\end{cases}
$$
\end{thm}

\subsection{Quadratic forms:upper bounds}\label{p-8}

The purpose of this subsection is to give a proof of the following proposition.
\begin{proposition}\label{h-5}
Fix $T>0,$ for  any $ N\geq 1,\alpha\in (0,1]$ and $\eta>0,$  there are positive constant $q_1=q_1(\alpha,N,T,\eta),q_2=q_2(\alpha,N,T,\eta)$ such that the following holds.
There exists a positive constant $\eps^*=\eps^*(\alpha,N,T,\eta)>0,$ such that, for any   $\eps\in (0,\eps^*]$, there exists a measurable set $\Omega_\eps^{*}=\Omega^{*}_\eps(\alpha,N,T,\eta)\subseteq \Omega$  and positive  constants  $C_1=C_1(\alpha,N,T,\eta),C_2=C_2(\alpha,N,T,\eta)$ such that
  \begin{eqnarray*}
    \mP((\Omega_\eps^*)^c )\leq C_1\eps ^{q_1}\exp{(\eta \|U_0\|^2)},
  \end{eqnarray*}
  and on the set  $\Omega_\eps^{*}$ one has,
  \begin{eqnarray*}
\langle \cM_{0,T}\phi,\phi\rangle \leq \eps \|\phi\|^2 ~\Rightarrow~ \langle Q_N(U)\phi,\phi\rangle \leq   C_2 \eps^{q_2}\|\phi\|^2
  \end{eqnarray*}
  which is valid for any  $\phi \in\cS_{\alpha,N}.$
\end{proposition}

Roughly speaking, this theorem suggests that the quadratic forms $Q_N$ are bound to have small eigenvalues on $\cS_{\alpha,N}$ with large probability once the Malliavin matrix $\mathcal{M}_{0,T}$ possesses a small eigenvalue.

Motivated by Section 4, we will adopt an iterative and inductive strategy to prove Theorem \ref{p-8}. To make this more precise, notice that
\begin{eqnarray*}
\langle \cM_{0,T}\phi,\phi\rangle=\sum_{\ell \in Z_0,m' \in\{0,1\}}(\alpha_\ell^{m'})^2\int_0^T \langle \sigma_\ell^{m'},\cK_{r,T}\phi\rangle^2\dif r.
\end{eqnarray*}
Therefore we start from that $\langle \cM_{0,T}\phi,\phi\rangle$ is small to deduce that $\langle \sigma_\ell^{m'},\cK_{r,T}\phi\rangle$ are small, which is the content of Lemma 5.3. Then by Lie brackets computation as suggested by Figure 4.1, we estimate progressively that $\langle Y_k^m(U),\cK_{r,T}\phi\rangle$, $\langle [Y_k^m(U),\sigma^{m'}_\ell],\cK_{r,T}\phi\rangle$ and $\langle \psi^m_{k+\ell},\cK_{r,T}\phi\rangle$ are all small, which are the contents of Lemma 5.4, Lemma 5.5 and Lemma 5.6 respectively. We also need to integrate all these results, since they only hold on different large sets, which is the content of Lemma 5.10.
Likewise, in the other direction we start from $\langle \psi_k^{m},\cK_{r,T}\phi\rangle$ are small to estimate progressively that $\langle \cY_k^m(U),\cK_{r,T}\phi\rangle$, $\langle [\cY_k^m(U),\sigma^{m'}_\ell],\cK_{r,T}\phi\rangle$ and $\langle \sigma^m_{k+\ell},\cK_{r,T}\phi\rangle$ are all small on some large sets, which are the contents of Lemma 5.7, Lemma 5.8 and Lemma 5.9 respectively.
Lemma 5.11 serves to integrate all these results.
The whole process is iterative and inductive so that we can tackle with successively larger finite dimensional subspace. To be specific, we refer the readers to Figure \ref{t1} for an illustration of the arguing structure in this subsection  that lead to the proof of Proposition \ref{h-5}.

\makeatletter

\def\@captype{figure}
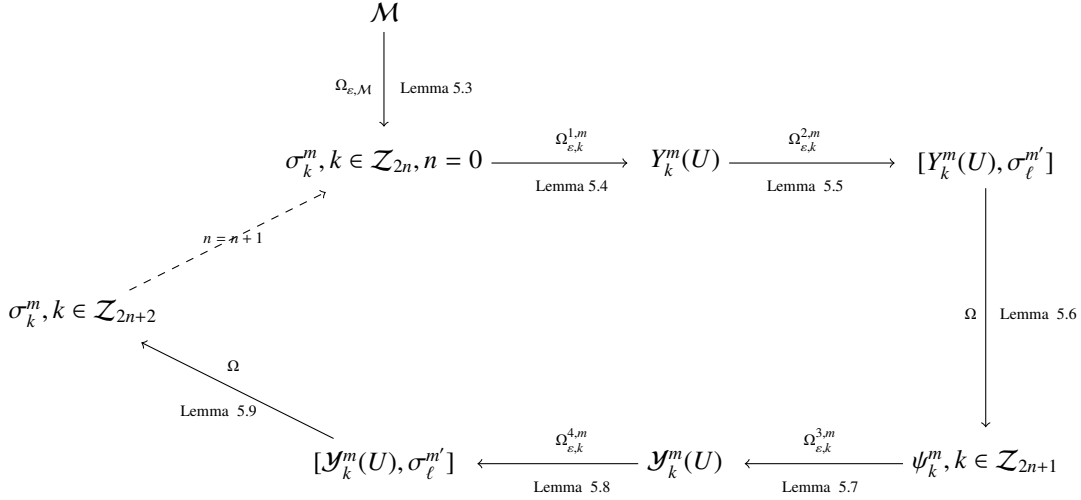
\begin{figure}[h]
\begin{tikzpicture}[shorten >=5pt,node distance=8cm]
\node(fs) at (4,6)   {$\cM$};
\node(ff) at (4,4) {$\sigma_{k}^{m}, k\in \cZ_{2n},n=0$};
\node(ef) at (8,4) { $Y_k^m(U)$};
\node(tf) at (12,4)  { $[Y_k^m(U),\sigma_{\ell}^{m'}]$ };
\node(fz) at (4,0) {$[\cY_k^m(U),\sigma_{\ell}^{m'}]$};
\node(ez) at (8,0) {$\cY_k^m(U)$};
\node(tz) at (12,0) {  $\psi_{k}^m,k\in \cZ_{2n+1}$};
\node(zt) at (0,2) {$\sigma_{k}^m,k\in\cZ_{2n+2}$};
 \path[line,->](fs) -- (ff);
 \node at (3.6,5){\tiny  $\Omega_{\eps,\cM}$};
 \node at (4.7,5){\tiny  Lemma \ref{h-10} };
 \path[line,->](ff) -- (ef);
  \node at (6.5,4.3){\tiny $\Omega_{\eps,k}^{1,m}$ };
 \node at (6.5,3.7){\tiny \text {Lemma} \ref{p-16} };
  \path[line,->](ef) -- (tf);
  \node at (9.6,4.3){\tiny $\Omega_{\eps,k}^{2,m}$ };
    \node at (9.6,3.7){\tiny \text{Lemma } \ref{p-17} };
  \path[line,->](tf) -- (tz);
    \node at (11.8,2){\tiny $\Omega $ };
    \node at (12.7,2){\tiny \text{Lemma } \ref{p-24} };
      \path[line,->](tz) -- (ez);
      \node at (9.8,0.3){\tiny $\Omega_{\eps,k}^{3,m} $ };
    \node at (9.8,-0.3){\tiny \text{Lemma } \ref{h-1} };
  \path[line,->](ez) -- (fz);
   \path[line,->](fz) -- (zt);
      \node at (2,1.3){\tiny $\Omega $ };
    \node at (1.8,0.7){\tiny \text{Lemma } \ref{p-25} };
    \draw[dashed, ->](zt) --node{} (ff);
       \node at (2,3){\tiny $n=n+1$ };

   \node at (6.5,0.3){\tiny $\Omega_{\eps,k}^{4,m} $ };
    \node at (6.5,-0.3){\tiny \text{Lemma } \ref{h-2} };
\end{tikzpicture}
\caption{\footnotesize  An illustration of the structure of the lemmas that leads to the proof of Proposition \ref{h-5}.    The  solid arrows  indicate that if one term is "small" then the other one "small" on a set of large measure(displayed up  or left of the arrow), where the meaning of "smallness" is made precise in each lemma. The  dashed  arrows   shows that the process is iterative.
In this figure, $m,m'\in \{0,1\}, \ell\in \cZ_0.$ One may notice the close relationship between Figure 4.1 and Figure 5.1.}
\label{t1}
\end{figure}

\newpage
\begin{lemma}\label{h-10}
For any $0<\eps<\eps_0(T)$ and every $\eta>0,$ there exists a set $\Omega_{\eps,\cM} $
and $C=C(\eta,T)$ with
\begin{eqnarray*}
  \mP(\Omega_{\eps,\cM} ^c)\leq C\exp\{\eta \|U_0\|^2\}\eps
\end{eqnarray*}
such that on the set $\Omega_{\eps,\cM}  $
  \begin{eqnarray}\label{p-52}
  \begin{split}
 \langle \cM_{0,T}\phi,\phi\rangle \leq \eps \|\phi\|^2 \Rightarrow~ \sup_{t\in [T/2,T]}| \langle \cK_{t,T}\phi,\sigma_{\ell}^{m'}\rangle| \leq   \eps^{1/8}\|\phi\|^2
\end{split}
  \end{eqnarray}
for each $\ell \in \cZ_0, m' \in \{0,1\}$ and $\phi\in H.$
\end{lemma}
\begin{proof}
Notice that
\begin{eqnarray*}
\langle \cM_{0,T}\phi,\phi\rangle=\sum_{\ell \in Z_0,m' \in\{0,1\}}(\alpha_\ell^{m'})^2\int_0^T \langle \sigma_\ell^{m'},\cK_{r,T}\phi\rangle^2\dif r.
\end{eqnarray*}
Define the function $  g_{\phi}(\cdot):[T/2,T]\rightarrow \mR^{+}$ as
\begin{eqnarray*}
  g_{\phi}(t):=\sum_{\ell \in  Z_0,m'\in\{0,1\}}(\alpha_\ell^{m'})^2\int_0^t \langle \sigma_\ell^{m'},\cK_{r,T}\phi\rangle^2\dif r,
\end{eqnarray*}
then
\begin{eqnarray*}
  g_{\phi}'(t)& =& \sum_{\ell \in  Z_0,m'\in \{0,1\}}(\alpha_\ell^{m'})^2\langle \sigma_\ell^{m'},\cK_{t,T}\phi\rangle^2,
\\   g_{\phi}^{''}(t)&=&2\sum_{\ell \in  Z_0, m'\in \{0,1\}}(\alpha_\ell ^{m'})^2\langle \sigma_\ell^{m'},\cK_{t,T}\phi\rangle \langle \sigma_\ell^{m'},\partial_t \cK_{t,T}\phi\rangle.
\end{eqnarray*}
Let
\begin{eqnarray*}
\Omega_{\eps,\cM}&=& \bigcap_{\phi\in H,\|\phi\|=1} \Big\{\sup_{t\in [T/2,T]} |g_{\phi}(t)|\geq \eps \text{  or } \sup_{t\in [T/2,T]} |g_{\phi}'(t)|\leq \eps^{1/4} \Big\}.
\end{eqnarray*}
Noticing  the definition of $\cZ_0$   and
\begin{eqnarray*}
 \langle \cK_{t,T}\phi,\sigma_\ell^{m'}\rangle= \langle \cK_{t,T}\phi,\sigma_{-\ell}^{m'}\rangle,
\end{eqnarray*}
then on $\Omega_{\eps,\cM},$ (\ref{p-52}) holds.
By Lemma \ref{lemma fght},  Lemma \ref{p-10} and (\ref{p-12}), we have
\begin{eqnarray*}
 \mP\left(\Omega_{\eps,\cM}^c\right)&\leq & \mP\left(\bigcup_{\phi\in H,\|\phi\|=1} \Big\{\sup_{t\in [T/2,T]} |g_{\phi}(t)|\leq \eps \text{  and } \sup_{t\in [T/2,T]} |g_{\phi}'(t)|\geq \eps^{1/4} \Big\},\right)
 \\ &\leq & C\eps \sum_{\ell \in \cZ_0,m'\in \{0,1\}}(\alpha_\ell^{m'})^4  \mE\[\sup_{\phi\in H, \|\phi\|=1}\sup_{t\in [T/2,T]}|\langle \sigma_\ell^{m'},\cK_{t,T}\phi\rangle \langle \sigma_\ell^{m'},\partial_t \cK_{t,T}\phi\rangle |^2\]
 \\ &\leq &C\eps\exp\{\eta \|U_0\|^2\}.
\end{eqnarray*}
\end{proof}

\begin{lemma}\label{p-16}
  Fix a certain $k\in \mZ^2, m\in \{0,1\}.$ For any $0<\eps<\eps_0(T)$ and $\eta>0,$ there  exists a set $\Omega_{\eps,k}^{1,m}$ and $C=C(k,\eta,T)$ with
  \begin{eqnarray*}
   \mP((\Omega_{\eps,k}^{1,m}) ^c)\leq C\exp\{\eta \|U_0\|^2\}\eps,
  \end{eqnarray*}
  such that on the set  $\Omega_{\eps,k}^{1,m},$   it holds that
  \begin{eqnarray}\label{p-4}
    \sup_{t\in [T/2,T]} |\langle \cK_{t,T}\phi,\sigma_k^m\rangle|\leq \eps \|\phi\|\Rightarrow  \sup_{t\in [T/2,T]} |\langle \cK_{t,T}\phi,Y_k^m(U)\rangle|\leq \eps^{1/10}\|\phi\|.
  \end{eqnarray}
\end{lemma}
\begin{proof}
  Define
   $ g_\phi(t):=\langle \cK_{t,T}\phi,\sigma_k^m\rangle, \forall t\in [0,T] $ and observe by (\ref{p-3}) that
   \begin{eqnarray*}
   g'_\phi(t)&=& \langle \cK_{t,T}\phi,[F(U),\sigma_k^m]\rangle
   \\ &=& \langle \cK_{t,T}\phi,Y_k^m(U)\rangle.
   \end{eqnarray*}
Let $\alpha=\frac{1}{4}$, and define
\begin{eqnarray*}
\Omega_{\eps,k}^{1,m}&=& \bigcap_{\phi\in H,\|\phi\|=1} \Big\{\sup_{t\in [T/2,T]} |g_{\phi}(t)|\geq \eps \text{   ~or  } \sup_{t\in [T/2,T]} |g_{\phi}'(t)|\leq \eps^{\alpha/2(1+\alpha)} \Big\}.
\end{eqnarray*}
Then on  $\Omega_{\eps,k}^{1,m},$ (\ref{p-4}) holds.
By \ref{lemma fght} we have
\begin{eqnarray*}
 \mP\left((\Omega_{\eps,k}^{1,m})^c\right)&\leq & \mP\left(\bigcup_{\phi\in H,\|\phi\|=1} \Big\{\sup_{t\in [T/2,T]} |g_{\phi}(t)|\leq \eps \text{  and } \sup_{t\in [T/2,T]} |g_{\phi}'(t)|\geq \eps^{\alpha/2(1+\alpha)} \Big\}\right)
 \\ &\leq & C\eps \mE\[\sup_{\phi\in H,\|\phi\|=1}\|g_{\phi}'\|_{C^\alpha[T/2,T]}^{2/\alpha}\].
\end{eqnarray*}

Since
\begin{eqnarray*}
   g'_\phi(t)&=&  \langle \cK_{t,T}\phi,Y_k^m(U)\rangle
   \\ &=& \langle \cK_{t,T}\phi,  Y_k^m(\bar{U})\rangle -\sum_{\ell \in Z_0,m' \in \{0,1\}}\alpha_{\ell}^{m'}\langle \cK_{t,T}\phi,  [Y_k^m(U),\sigma_\ell^{m'}]\rangle W^{\ell,m'},
 \end{eqnarray*}
there follows
\begin{eqnarray*}
\|g_{\phi}'\|_{C^\alpha[T/2,T]}& \leq & C\sup_{t\in [T/2,T]} |\partial_t  \langle \cK_{t,T}\phi,Y_k^m(\bar{U})\rangle|+C  \sum_{\ell  \in Z_0,m'\in \{0,1\}} \sup_{t\in [T/2,T]} | \langle  \cK_{t,T}   \phi,  [Y_k^m(U),\sigma_\ell^{m'}]\rangle |\cdot   |W^{\ell,m'}|_{C^\alpha[T/2,T]}
\\ &&+ C  \sum_{\ell \in Z_0,m'\in \{0,1\}}  | \langle  \cK_{t,T}   \phi,  [Y_k^m(U),\sigma_\ell^{m'}]\rangle  |_{C^\alpha[T/2,T]} \cdot  \sup_{t\in [T/2,T]}|W^{\ell,m'}_t|.
\end{eqnarray*}
Therefore, by Lemma \ref{p-14} one gets
\begin{eqnarray*}
\mE\[\sup_{\phi\in H,\|\phi\|=1}\|g_{\phi}'\|_{C^\alpha[T/2,T]}^{2/\alpha}\]
&&\leq C(\eta,\alpha)\exp{(\eta \|U_0\|^2)}.
\end{eqnarray*}
\end{proof}

\begin{lemma}\label{p-17}
  Fix a certain $k\in \mZ_+^2.$ For any $0<\eps<\eps_0(T)$ and $\eta>0,$ there  exists a set $\Omega_{\eps,k}^{2,m}$ and $C=C(\eta,T,k)$ with
  \begin{eqnarray*}
   \mP((\Omega_{\eps,k}^{2,m}) ^c)\leq C\exp\{\eta \|U_0\|^2\}\eps^{1/9},
  \end{eqnarray*}
  such that on the set  $\Omega_{\eps,k}^{2,m},$ it holds
  \begin{eqnarray*}
    \sup_{t\in [T/2,T]} |\langle \cK_{t,T}\phi,Y_k^m(U)\rangle|\leq \eps \|\phi\|\Rightarrow  \sup_{\ell \in \cZ_0,m' \in \{0,1\}}\sup_{t\in [T/2,T]} |\alpha_\ell^{m'}|\cdot   |\langle \cK_{t,T}\phi,[Y_{k}^m(U),\sigma_\ell^{m'}]\rangle|\leq \eps^{1/3}\|\phi\|.
  \end{eqnarray*}
\end{lemma}

\begin{proof}
By expanding one finds
\begin{eqnarray*}
\langle \cK_{t,T}\phi,Y_k^m(U)\rangle& =& \langle \cK_{t,T}\phi,Y_k^m(\bar{U})\rangle
-\sum_{\ell \in  Z_0,m' \in \{0,1\}}\alpha_\ell^{m'} \langle \cK_{t,T}\phi, [Y_k^m(U),\sigma_\ell^{m'}]\rangle W^{\ell,m'}.
\end{eqnarray*}
For $\alpha \in \{0,1\},\phi\in H $, we recall that
\begin{eqnarray*}
  \cN_\alpha(\phi)=\max_{\ell \in  Z_0,m'\in\{0,1\}}\Big\{\| \langle \cK_{t,T}\phi,Y_k^m(\bar{U})\rangle\|_{C^\alpha}, ~~~~~ |\alpha_\ell^{m'}|\cdot \| \langle   \cK_{t,T}\phi, [Y_k^m(U),\sigma_\ell^{m'}]   \rangle\|_{C^\alpha}\Big\}.
\end{eqnarray*}
Then by Theorem \ref{theorem 6.4}, there exists a set $\Omega^{\#}_\eps$ such that
$$
\mP( (\Omega^{\#}_\eps)^c)\leq C\eps,
$$
and
on $\Omega^{\#}_\eps$
\begin{eqnarray*}
   \sup_{t\in [T/2,T]} |\langle \cK_{t,T}\phi,Y_k^m(U)\rangle|\leq \eps \|\phi\|\Rightarrow
   \left\{
   \begin{split}
     &  \text{either } \cN_0(\phi)\leq \eps^{1/3},
     \\
     &\text{or  } \cN_1(\phi)\geq \eps^{-1/9}.
   \end{split}
   \right.
\end{eqnarray*}
Let
\begin{eqnarray*}
\Omega_{\eps,k}^{2,m}: =\Omega^{\#}_\eps \cap \cap_{\phi \in H, \|\phi\|=1 } \{\cN_1(\phi)< \eps^{-1/9}  \}.
\end{eqnarray*}
Then this lemma follows from
 Lemma \ref{p-14}, (\ref{p-22}) and the fact
 \begin{eqnarray*}
|\langle \cK_{t,T}\phi,[Y_{k}^m(U),\sigma_\ell^{m'}]\rangle|
  &=& |\langle \cK_{t,T}\phi,[Y_{k}^m(U),\sigma_{-\ell}^{m'}]\rangle|.
\end{eqnarray*}
\end{proof}

\begin{lemma}\label{p-24}
For any $n\in \mN$  there exists a constant $C_n$ such that for any  $k\in \cZ_{2n},$
  \begin{eqnarray*}
 \sup_{\ell \in \cZ_0,m,m' \in \{0,1\}}\sup_{t\in [T/2,T]} |\langle \cK_{t,T}\phi,[Y_{k}^m(U),\sigma_\ell^{m'}]\rangle|\leq \eps \|\phi\|
\end{eqnarray*}
implies
\begin{eqnarray*}
\sup_{\ell \in \cZ_0, \ell \notin \{k,-k\} }\sup_{m\in \{0,1\}}\sup_{t\in [T/2,T]} |\langle \cK_{t,T}\phi,\psi_{k+\ell}^m]\rangle|\leq C_n \eps \|\phi\|
  \end{eqnarray*}
  with probability one.
\end{lemma}
\begin{proof}
  It directly follows from Lemma \ref{p-18} and (\ref{p-49}).
\end{proof}

\begin{lemma}\label{h-1}
  Fix some $k\in \mZ^2, m\in \{0,1\}.$ For any $0<\eps<\eps_0(T)$ and $\eta>0,$ there  exists a set $\Omega_{\eps,k}^{3,m}$ and $C=C(\eta,k,T)$ with
  \begin{eqnarray*}
   \mP((\Omega_{\eps,k}^{3,m}) ^c)\leq C\exp\{\eta \|U_0\|^2\}\eps,
  \end{eqnarray*}
  such that on the set  $\Omega_{\eps,k}^{3,m},$ for each $m\in \{0,1\}$, it holds
  \begin{eqnarray}\label{pp-4}
    \sup_{t\in [T/2,T]} |\langle \cK_{t,T}\phi,\psi_k^m\rangle|\leq \eps \|\phi\|\Rightarrow  \sup_{t\in [T/2,T]} |\langle \cK_{t,T}\phi, \big[F(U),\psi_k^{m}\big]\rangle|\leq \eps^{1/10}\|\phi\|.
  \end{eqnarray}
\end{lemma}

\begin{proof}
  Define
   $ g_\phi(t):=\langle \cK_{t,T}\phi,\psi_k^m\rangle$ and observe by (\ref{p-3}) that
   \begin{eqnarray*}
   g'_\phi(t)&=& \langle \cK_{t,T}\phi,[F(U),\psi_k^m]\rangle.
   \end{eqnarray*}

Let $\alpha=\frac{1}{4}$, and define
\begin{eqnarray*}
\Omega_{\eps,k}^{3,m}&=& \bigcap_{\phi\in H,\|\phi\|=1} \Big\{\sup_{t\in [T/2,T]} |g_{\phi}(t)|\geq \eps \text{  or } \sup_{t\in [T/2,T]} |g_{\phi}'(t)|\leq \eps^{\alpha/2(1+\alpha)} \Big\}.
\end{eqnarray*}
Then on  $\Omega_{\eps,k}^{3,m},$ (\ref{pp-4}) holds.
By Theorem \ref{lemma fght} we have
\begin{eqnarray*}
 \mP\left((\Omega_{\eps,k}^{3,m})^c\right)&\leq & \mP\left(\bigcup_{\phi\in H,\|\phi\|=1} \Big\{\sup_{t\in [T/2,T]} |g_{\phi}(t)|\leq \eps \text{  and } \sup_{t\in [T/2,T]} |g_{\phi}'(t)|\geq \eps^{\alpha/2(1+\alpha)} \Big\}\right)
 \\ &\leq & C\eps \mE\[\sup_{\phi\in H, \|\phi \|=1}\|g_{\phi}'\|_{C^\alpha[T/2,T]}^{2/\alpha}\].
\end{eqnarray*}

Since,
\begin{eqnarray*}
   g'_\phi(t)&=&  \langle \cK_{t,T}\phi,\cY_k^m(U)\rangle
   \\ &=& \langle \cK_{t,T}\phi,  \cY_k^m(\bar{U})\rangle -\sum_{\ell  \in Z_0,m' \in \{0,1\}}\alpha_\ell^{m'}\langle \cK_{t,T}\phi, [\cY_k^m(U),\sigma_\ell^{m'}]\rangle W^{\ell,m'},
 \end{eqnarray*}
there follows
\begin{eqnarray*}
\|g_{\phi}'\|_{C^\alpha[T/2,T]}& \leq & C \sup_{t\in [T/2,T]}|\partial_t  \langle \cK_{t,T}\phi,\cY_k^m(\bar{U})\rangle|+C  \sum_{\ell  \in Z_0,m'\in \{0,1\}} \sup_{t\in [T/2,T]} | \langle  \cK_{t,T}   \phi,  [\cY_k^m(U),\sigma_\ell^{m'}]\rangle  |\cdot |W^{\ell,m'}|_{C^\alpha[T/2,T]}
\\ &&+ C  \sum_{\ell \in  Z_0,m'\in \{0,1\}}  | \langle  \cK_{t,T}   \phi,  [\cY_k^m(U),\sigma_\ell^{m'}]\rangle  |_{C^\alpha[T/2,T]}\cdot  \sup_{t\in [T/2,T]}|W^{\ell,m'}(t)|.
\end{eqnarray*}
By Lemma \ref{p-14} one gets
\begin{eqnarray*}
\mE\[\sup_{\phi\in H,\|\phi\|=1}\|g_{\phi}'\|_{C^\alpha[T/2,T]}^{2/\alpha}\]
&&\leq C(\eta,\alpha)\exp{(\eta \|U_0\|^2)}.
\end{eqnarray*}
\end{proof}

\begin{lemma}\label{h-2}
  Fix some $k \in \mZ_+^2, m\in \{0,1\}.$ For any $0<\eps<\eps_0(T)$ and $\eta>0,$ there  exists a set $\Omega_{\eps,k}^{4,m}$ and $C=C(k, \eta,T)$ with
  \begin{eqnarray*}
   \mP(( \Omega_{\eps,k}^{4,m}) ^c)\leq C \exp\{\eta \|U_0\|^2\}\eps,
  \end{eqnarray*}
  such that on the set  $\Omega_{\eps,k}^{4,m} ,$ it holds
  \begin{eqnarray*}
   \nonumber  \sup_{t\in [T/2,T]} |\langle \cK_{t,T}\phi,\cY_k^m(U)\rangle|\leq \eps \|\phi\|\Rightarrow  \sup_{\ell \in \cZ_0,m'  \in \{0,1\}}\sup_{t\in [T/2,T]} |\alpha_\ell^{m'}|\cdot \big|\big\langle \cK_{t,T}\phi,\big[\cY_k^m(U),\sigma_\ell ^{m'}\big]\big\rangle\big|\leq \eps^{1/3}\|\phi\|.
    \\
  \end{eqnarray*}
\end{lemma}
\begin{proof}
By expanding one finds
\begin{eqnarray*}
\langle \cK_{t,T}\phi,\cY_k^m(U)\rangle& =& \langle \cK_{t,T}\phi,\cY_k^m(\bar{U})\rangle
-\sum_{\ell  \in  Z_0,m' \in \{0,1\}}\alpha_\ell^{m'} \langle \cK_{t,T}\phi, [\cY_k^m(U),\sigma_\ell^{m'}]\rangle W^{\ell,m'}.
\end{eqnarray*}
Recall that, for $\alpha \in \{0,1\},\phi\in H $
\begin{eqnarray*}
  \cM_\alpha(\phi):=\max_{\ell \in  Z_0,m' \in\{0,1\}}\Big\{\| \langle \cK_{t,T}\phi,\cY_k^m(\bar{U})\rangle\|_{C^\alpha},  |\alpha_\ell^{m'}|\cdot \| \langle   \cK_{t,T}\phi, [\cY_k^m(U),\sigma_\ell^{m'}]   \rangle\|_{C^\alpha}\Big\}.
\end{eqnarray*}
Then by Theorem \ref{theorem 6.4}, there exists a set $\Omega^{\#}_\eps$ such that
$$
\mP( (\Omega^{\#}_\eps)^c)\leq C\eps,
$$
and
on $\Omega^{\#}_\eps$
\begin{eqnarray*}
   \sup_{t\in [T/2,T]} |\langle \cK_{t,T}\phi,\cY_k^m(U)\rangle|\leq \eps \|\phi\|\Rightarrow
   \left\{
   \begin{split}
     &  \text{either } \cM_0(\phi)\leq \eps^{1/3},
     \\
     &\text{or  } \cM_1(\phi)\geq \eps^{-1/9}.
   \end{split}
   \right.
\end{eqnarray*}
Let
\begin{eqnarray*}
\Omega_{\eps,k}^{4,m}=\Omega^{\#}_\eps \cap \cap_{\phi \in H, \|\phi\|=1 } \{\cM_1(\phi)< \eps^{-1/9}  \}.
\end{eqnarray*}
Then this lemma follows from
 Lemma \ref{p-14}, (\ref{p-22}) and the fact
 \begin{eqnarray*}
|\langle \cK_{t,T}\phi,[\cY_{k}^m(U),\psi_\ell^{m'}]\rangle|
  &=& |\langle \cK_{t,T}\phi,[\cY_{k}^m(U),\psi_{-\ell}^{m'}]\rangle|.
\end{eqnarray*}
\end{proof}

\begin{lemma}\label{p-25}
For any $n\in \mN$  there exists a constant $C=C(n)$ such that for any  $k\in \cZ_{2n+1},$
  \begin{eqnarray*}
 \sup_{\ell \in \cZ_0,m,m' \in \{0,1\}}\sup_{t\in [T/2,T]} |\langle \cK_{t,T}\phi,[\cY_{k}^m(U),\sigma_\ell^{m'}]\rangle|\leq \eps \|\phi\|
\end{eqnarray*}
implies
\begin{eqnarray*}
\sup_{\ell \in \cZ_0, \ell \notin \{k,-k\} }\sup_{m\in \{0,1\}}\sup_{t\in [T/2,T]} |\langle \cK_{t,T}\phi,\sigma_{k+\ell}^m]\rangle|\leq C \eps \|\phi\|
  \end{eqnarray*}
  with probability one.
\end{lemma}
\begin{proof}
  It directly follows from  (\ref{p-50}) and  Lemma \ref{h-3}.
\end{proof}

\begin{lemma}\label{h-11}
For any $n \in \mN$, and $ q_{2n},C_{2n}>0,$ there exist   $p_{2n+1}, q_{2n+1},C_{2n+1}>0$,   a set $\Omega_{\eps,2n}$ and a constant   $C=C(n, \eta,T)$ with
  \begin{eqnarray*}
   \mP(\Omega_{\eps,2n} ^c)\leq C \exp\{\eta \|U_0\|^2\}\eps^{p_{2n+1}},
  \end{eqnarray*}
  such that on the set  $\Omega_{\eps,2n},$ it holds
    \begin{eqnarray*}
    && \sum_{k \in \cZ_{2n}, m\in \{0,1\} }\sup_{t\in [T/2,T]} |\langle \cK_{t,T}\phi,\sigma_k^m\rangle| \leq C_{2n}\eps^{q_{2n}} \|\phi\|
    \\ && \Rightarrow   \sum_{k  \in \cZ_{2n+1}, m\in \{0,1\} }\sup_{t\in [T/2,T]} |\langle \cK_{t,T}\phi,\psi_k^m\rangle|\leq C_{2n+1}\eps^{q_{2n+1}} \|\phi\|.
  \end{eqnarray*}
\end{lemma}
\begin{proof}
  For any  $k \in\cZ_{2n}, m\in \{0,1\}$, by Lemma \ref{p-16},  there exist $p_{2n}', C_{2n+1}', q_{2n+1}'$  and  a set   $\Omega_{\eps,k}^{1,m}$ such that  on $\Omega_{\eps,k}^{1,m}$,
  \begin{eqnarray*}
     \sup_{t\in [T/2,T]} |\langle \cK_{t,T}\phi,\sigma_k^m\rangle|\leq C_{2n}\eps^{q_{2n}} \|\phi\|\Rightarrow  \sup_{t\in [T/2,T]} |\langle \cK_{t,T}\phi,Y_k^m(U)\rangle|\leq C_{2n+1}'\eps^{q_{2n+1}'}\|\phi\|,
  \end{eqnarray*}
  and
  \begin{eqnarray*}
    \mP(  (\Omega_{\eps,k}^{1,m})^c )\leq  C\exp\{\eta \|U_0\|^2\}\eps^{p_{2n}'}.
  \end{eqnarray*}

Next by Lemma \ref{p-17},  there exist $p_{2n}, C_{2n+1}, q_{2n+1}$  and  a set   $\Omega_{\eps,k}^{2,m}$ such that  on $\Omega_{\eps,k}^{2,m}$,
 \begin{eqnarray*}
&&  \sup_{t\in [T/2,T]} |\langle \cK_{t,T}\phi,Y_k^m(U)\rangle|\leq C_{2n+1}'\eps^{q_{2n+1}'}\|\phi\|
 \\ &&  \Rightarrow \sup_{\ell \in \cZ_0,m'  \in \{0,1\}}\sup_{t\in [T/2,T]} |\langle \cK_{t,T}\phi,[Y_{k}^m(U),\sigma_\ell^{m'}]\rangle|\leq  C_{2n+1}\eps^{q_{2n+1}}\|\phi\|,
 \end{eqnarray*}
 and
   \begin{eqnarray*}
    \mP(  (\Omega_{\eps,k}^{2,m})^c )\leq  C\exp\{\eta \|U_0\|^2\}\eps^{p_{2n}}.
  \end{eqnarray*}
  Set
  \begin{eqnarray*}
\Omega_{\eps,2n}=\cap_{k  \in\cZ_{2n}, m\in\{0,1\} }\[ \Omega_{\eps,k}^{1,m} \cap \Omega_{\eps,k}^{2,m} \] ,
  \end{eqnarray*}
 then this lemma follows from  (\ref{p-49}), Lemma \ref{p-24} and Lemma \ref{p-18}.
\end{proof}

\begin{lemma}\label{h-12}
For any $n\in \mN$, and $q_{2n+1},C_{2n+1}>0,$ there exist   $p_{2n+2}, q_{2n+2},C_{2n+2}>0$,   a set $\Omega_{\eps,2n+1}$ and a constant   $C=C(n, \eta,T)$ with
  \begin{eqnarray*}
   \mP(\Omega_{\eps,2n+1} ^c)\leq C \exp\{\eta \|U_0\|^2\}\eps^{p_{2n+2}},
  \end{eqnarray*}
  such that on the set  $\Omega_{\eps,2n+1},$ it holds
    \begin{eqnarray*}
    && \sum_{k  \in \cZ_{2n+1}, m\in \{0,1\} }\sup_{t\in [T/2,T]} |\langle \cK_{t,T}\phi,\psi_k^m\rangle| \leq C_{2n+1}\eps^{q_{2n+1}} \|\phi\|
    \\ && \Rightarrow   \sum_{k  \in \cZ_{2n+2}, m\in \{0,1\} }\sup_{t\in [T/2,T]} |\langle \cK_{t,T}\phi,\sigma_k^m\rangle|\leq C_{2n+2}\eps^{q_{2n+2}} \|\phi\|.
  \end{eqnarray*}
\end{lemma}

\begin{proof}
By Lemma \ref{h-1}, for any $m\in \{0,1\}, k\in \cZ_{2n+1},\eps>0$, there exist set $\Omega_{\eps,k}^{3,m}$ and $p_{2n+2}', q_{2n+2}',C_{2n+2}'>0$ such that
  \begin{eqnarray*}
   \mP((\Omega_{\eps,k}^{3,m}) ^c)\leq C_{2n+2}'\exp\{\eta \|U_0\|^2\}\eps^{p_{2n+2}'},
  \end{eqnarray*}
  such that on the set  $\Omega_{\eps,k}^{3,m}$, it holds
  \begin{eqnarray*}
    \sup_{t\in [T/2,T]} |\langle \cK_{t,T}\phi,\psi_k^m\rangle|\leq \eps \|\phi\|\Rightarrow  \sup_{t\in [T/2,T]} |\langle \cK_{t,T}\phi,\cY_k^m(U)\rangle|\leq \eps^{q_{2n+2}'}\|\phi\|.
  \end{eqnarray*}

Next by Lemma  \ref{h-2},
 for any $m\in \{0,1\}, k\in \cZ_{2k+1},\eps>0$, there exist set $\Omega_{\eps,k}^{4,m}$ and $p_{2n+2}, q_{2n+2},C_{2n+2}>0$ such that
  \begin{eqnarray*}
   \mP((\Omega_{\eps,k}^{4,m}) ^c)\leq C_{2n+2}\exp\{\eta \|U_0\|^2\}\eps^{p_{2n+2}},
  \end{eqnarray*}
  and on the set  $\Omega_{\eps,k}^{4,m}$, it holds
  \begin{eqnarray*}
 \sup_{t\in [T/2,T]} |\langle \cK_{t,T}\phi,\cY_k^m(U)\rangle|\leq \eps^{q_{2n+2}'} \|\phi\|\Rightarrow  \sup_{\ell \in \cZ_0,m'  \in \{0,1\}}\sup_{t\in [T/2,T]} \big|\big\langle \cK_{t,T}\phi,\big[\cY_k^m(U),\sigma_\ell^{m'}\big]\big\rangle\big|\leq \eps^{q_{2n+2}}\|\phi\|.
  \end{eqnarray*}

  Set
  \begin{eqnarray*}
\Omega_{\eps,2n+1}=\cap_{k \in\cZ_{2n+1}, m\in\{0,1\} }\[ \Omega_{\eps,k}^{3,m} \cap \Omega_{\eps,k}^{4,m} \] ,
  \end{eqnarray*}
 then this lemma follows from  (\ref{p-50}), Lemma \ref{p-25} and Lemma \ref{h-3}.
\end{proof}

\begin{proof}[\textbf{Proof of Proposition \ref{h-5}}]
First,  we recall the definition of $\Omega_{\eps, \cM}$ from Lemma \ref{h-10} and let $C_0=1,q_0=\frac{1}{8}$. Then for any $n\in \mN,~$ just after constants $C_{2n},q_{2n}$ are fixed,
we set $p_{2n+1},q_{2n+1},C_{2n+1},\Omega_{\eps,2n}$ by Lemma \ref{h-11} and $p_{2n+2},q_{2n+2},C_{2n+2}, \Omega_{\eps,2n+1}$ by Lemma \ref{h-12}.  Recursively, for any $n\in \mN,$ $\Omega_{\eps,n},C_n, p_n, q_n$ are  well chosen.

Let
\begin{eqnarray*}
\Omega^{*}_{\eps}= \Omega_{\eps, \cM}\cap \cap_{n=0}^{2N+1} \Omega_{\eps, n}.
\end{eqnarray*}
Integrating  Lemma \ref{h-11} and Lemma \ref{h-12} with Lemma \ref{h-10},
we have for some positive constants   $p_N^*, q_N^*$,
 $C=C(\eta,T,N)$  that
  \begin{eqnarray*}
    \mP((\Omega_\eps^*)^c )\leq C\eps ^{p_N^*}\exp{(\eta \|U_0\|^2)},
  \end{eqnarray*}
  and on the set  $\Omega_\eps^{*}$
  \begin{eqnarray*}
\langle \cM_{0,T}\phi,\phi\rangle \leq \eps \|\phi\|^2 ~\Rightarrow~ \langle Q_N\phi,\phi\rangle \leq  C \eps^{q_N^*}\|\phi\|^2,
  \end{eqnarray*}
  which is valid for any  $\phi \in\cS_{\alpha,N}.$  The proof is finished.
\end{proof}

\subsection{Proof of Theorem \ref{p-9}}\label{p-59}
Now we are in a position to prove Theorem   \ref{p-9}.
\begin{proof}
Set  $\Omega_\eps=\Omega_\eps^*$, which is given by Proposition \ref{p-8}.
Let $\eps^{*}$  be a constant such that  for any $\eps\in (0,\eps^*]$
\begin{eqnarray}\label{p-56}
\frac{\alpha}{2}>C_2 \eps^{q_2}.
\end{eqnarray}
Again $C_2,q_2$ are constants given by  Proposition  \ref{p-8}.

First, by Proposition \ref{p-8}, (\ref{p-54}) holds.

Next on the set  $\Omega_\eps=\Omega_\eps^*$,  for any  $\phi \in\cS_{\alpha,N}$ satisfying
\begin{eqnarray*}
  \langle \cM_{0,T}\phi,\phi\rangle< \eps \|\phi\|^2,
\end{eqnarray*}
Proposition \ref{p-53} and Proposition \ref{p-8} imply
\begin{eqnarray*}
\frac{\alpha}{2}\|\phi\|^2\leq \langle Q_N\phi,\phi\rangle \leq C_2\eps^{q_2}\|\phi\|^2,
\end{eqnarray*}
which contradicts	  with   (\ref{p-56}).  Therefore, (\ref{p-55}) holds on the set
$\Omega_\eps$.
\end{proof}
Once Proposition \ref{h-5} is established,  one can translate spectral bounds on the Malliavin matrix $\mathcal{M}$ to the estimate on $\nabla P_t\Phi$. This constitutes the main content of the next proposition, and since the Malliavin matrix $\mathcal{M}$ only prove to be nondegenerate on finite dimensional cones, gradient estimates are bound to be in a asymptotic form, which leads Hairer \cite{martin} to introduce the celebrated conception of asymptotic strong Feller.

\begin{proposition}\label{p-28}
  For some $\gamma_0>0$ and every $\eta>0, U_0\in H$, the Markov semigroup $\{P_t\}_{t\geq 0}$ defined by (\ref{p-26}) satisfies the following estimate
  \begin{eqnarray*}
    \|\nabla P_t\Phi(U_0)\|\leq C\exp{(\eta \|U_0\|^2)}\left(\sqrt{P_t(|\Phi|^2)(U_0)}+e^{-\gamma_0t}
    \sqrt{P_t(\|\nabla\Phi\|^2)(U_0)}\right)
  \end{eqnarray*}
  for every $t\geq 0$ and $\Phi \in C_b(H)$, where $C=C(\eta,\gamma_0)$ is independent of $t$ and $\Phi.$
\end{proposition}

\begin{proof}
Ever since \cite{martin}  the proof of this type of gradient inequality have been attached great importance to and improved all along. Now the method to prove it is more or less standard.  Broadly speaking, supplied with moment estimates of $U,J_{s,t}\xi, \cK_{s,t}\xi,$ $J_{s,t}^{(2)}(\xi,\xi')$ listed in  Section 2, one need to formulate a control problem through the Malliavin integration by parts formula, then do some decay estimates adopting an iterative construction with the aid of Lemma 3.4, Lemma 3.5, Lemma 3.6. We refer the readers to \cite{FGRT,martin,Hairer02,Hairer} and omit the details.

\end{proof}

\section{Proof of Theorem  \ref{p-27}}

Our strategy in this section is to apply \cite[Theorem 3.4]{Hairer02} and \cite[Theorem 2.1]{KW}, separately, to draw the conclusion of mixing rates and central limit theorem. Since it is very straightforward and similar to the proof of \cite[Theorem 2.3]{FGRT}, we will sketch our arguments. Before carrying them out, we need to introduce a type of 1-Wasserstein distance. Referring to Lemma \ref{p-30} (1) to fix some $\eta^*>0$, then for any $\eta \in (0, \eta^*], r\in (0,1],$  define the metric  $\rho_{r}$ on $H$  by
\begin{eqnarray}\label{peng-1}
  \rho_{r}(U_1,U_2):=\inf_{\gamma}\int_0^1 \exp{(\eta r\|\gamma(t)\|^2 )}\|\gamma'(t) \|\dif t,
\end{eqnarray}
where the infimum  runs over all paths $\gamma$ such that $\gamma(0)=U_1$ and $\gamma(1)=U_2.$ For brevity of notation, we set $\rho:=\rho_1.$

\begin{proof}[\textbf{Proof of Theorem  \ref{p-27}}]
(a)
Ito's formula yields that
  \begin{eqnarray*}
   \|U_t\|^2-\|U_0\|^2+2\int_0^t \|\Lambda^\alpha u_s\|^2\dif s +2\int_0^t \|\Lambda^\beta b_s\|^2\dif s &=& \cE_0t +2 \int_0^t \langle b_s, \mathcal{Q}_b\dif W_s\rangle,
  \end{eqnarray*}
then  it follows that
\begin{eqnarray*}
  \frac{1}{T}\mE \int_0^T \|U_t\|_{H^1}\dif t\leq  \frac{\|U_0\|^2}{T}+\cE_0.
\end{eqnarray*}
By the classical Krylov-Bogoliubov averaging method, one arrives at that there exists an invariant measure for the semigroup $P_t.$

(b) Finding a strong type of Lyapunov structure: Let $\kappa=\frac{3}{2}, r_0=\frac{1}{4},$ and $\eta'=\frac{1}{4}\cdot \frac{\eta}{2}\cdot e^{-1/2}$, by lemma  \ref{p-29} we have
\begin{eqnarray*}
  \|J_t\xi\|\leq C\exp{(\eta'\int_0^t \|U_s\|_{H^1}^2\dif s)}.
\end{eqnarray*}
Then  by Lemma \ref{p-30}, for $r\in [r_0,2\kappa]$ and $t\in [0,1]$ we get
\begin{eqnarray*}
  \mE \[\exp{(r\eta \|U_t\|^2)}(1+\|J_t\xi\|)\]
  & \leq  &  C  \mE \[\exp{(r\eta \|U_t\|^2+\eta'\int_0^t \|U_s\|_{H^1}^2)}\]
  \\ & \leq &  C  \mE \[\exp{(r\eta \|U_t\|^2+\frac{\eta r}{2}e^{- t/2}\int_0^t \|U_s\|_{H^1}^2)}\]
  \\ &\leq & C\exp\{\eta r |U_0|^2e^{-\frac{t}{ 2} }\}.
\end{eqnarray*}
Therefore, \cite[Assumption 4]{martin}
is verified with  $\kappa=\frac{3}{2},\eta\in (0,\frac{1}{3}\eta^*), r_0=\frac{1}{4},$
\begin{eqnarray*}
 && V_*(x)=\exp{(\frac{3}{4}\eta x^2)},~~~ \forall x\in \mR,
\\ &&  V^*(x)=\exp{(\frac{17}{16}\eta x^2)},~~~ \forall x\in \mR,
\\  &&  V(x)=\exp((\eta x^2)) ,~~~ \forall x\in \mR,
\\ && \xi(t)=e^{-\frac{t}{2}},~~~t\in [0,1],
 \end{eqnarray*}
which suggests the desirable Lyapunov structure.

(c) Gradient inequality on the Markov semigroup:  This is just reemphasizing. By Proposition \ref{p-28},   for some $\gamma_0>0$ and every $\eta>0, U_0\in H$, the Markov semigroup $\{P_t\}_{t\geq 0}$ defined by (\ref{p-26}) satisfies the following estimate
  \begin{eqnarray*}
    \|\nabla P_t\Phi(U_0)\|\leq C\exp{(\eta \|U_0\|^2)}\left(\sqrt{P_t(|\Phi|^2)(U_0)}+e^{-\gamma_0t}
    \sqrt{P_t(\|\nabla\Phi\|^2)(U_0)}\right)
  \end{eqnarray*}
  for every $t\geq 0$ and $\Phi \in C_b(H)$, where $C=C(\eta,\gamma_0)$ is independent of $t$ and $\Phi.$

(d) What we need now is to establish a relatively weak form of
irreducibility, i.e., for any  $\varrho,\eps>0, r\in (0,1)$, there exists $T^*=T^*(\varrho,r,\eps)$ such that for any $T>T^*,$
\begin{eqnarray}
  \inf_{\|U_1\|,\|U_2\|\leq \varrho } \sup_{\Gamma \in \cC(P_T^*\delta_{U_1},P_T^*\delta_{U_2})} \Gamma\{(U',U'')\in H\times H:\rho_r(U',U'')<\eps\}>0,
\end{eqnarray}
where $\delta_U$ is the dirac measure concentrated on $U$ and  $\cC(\mu_1,\mu_2)$ denotes the set of all coupling measures $\pi$
on $H \times H$ such that $\pi(A \times H)=\mu_1(A)$ and
$\pi(H\times A)=\mu_2(A)$ for every Borel set~$ A \subset H.$

 In fact, this can be deduced immediately by irreducibility in general sense, which is, for any  $\varrho,\eps>0$ there exists $T_*=T_*(\varrho,\eps)\geq 0$ such that
\begin{eqnarray}
 \inf_{\|U_0\|\leq \varrho } P_T(U_0, \{U\in H,\|U\|\leq \eps\})>0,
\end{eqnarray}
for any $T>T_*.$ Utilizing the dissipativity of deterministic system and properties of Gaussian distribution, this can be proved following the classical arguments(c.f. \cite{G-1996}).

It is common sense that (a)-(d) implicates there exits a unique invariant measure $\mu_*$ for $P_t$.

For every $\Phi\in \mathcal{O}_\eta$, one can show that
\begin{eqnarray*}
  \int \Phi(U)\dif \mu_*(U)\leq C \|\Phi\|_\eta,
\end{eqnarray*}
for some constant $C$ dependent on $\eta$. From this, Applying \cite[Theorem 3.4]{Hairer02} yields that the unique invariant measure $\mu_*$ is exponentially mixing.

(e) To apply \cite[Theorem 2.1]{KW}, the following inequality is critical
\begin{eqnarray}\label{p-34}
  \int [\rho(0,U)]^3P_t(U_0,\dif U)&\leq & C\exp{(\eta^* \|U_0\|^2)},
\end{eqnarray}
where the constant $C$ is independent of $U_0$ and $t\geq 0.$ With the definition of $\rho$, this can be easily established from Lemma \ref{p-30}. The proof is finished.
\end{proof}

\appendix





 \section*{Acknowledgements}


The authors  thank the anonymous referees for their  valuable comments and suggestions.



\bibliographystyle{elsarticle-num}



\end{document}